\documentclass{amsart}

\usepackage{amsmath}
\usepackage{amssymb}
\usepackage{amsfonts}
\usepackage{graphicx}
\usepackage{mathrsfs}
\usepackage[margin=1.15in]{geometry}
\usepackage{amscd}
\usepackage[all]{xy}
\usepackage{hyperref}
\usepackage{tikz-cd}

\title[Artin-Nagata Properties, Minimal Multiplicities, and Depth of Fiber Cones]{Artin-Nagata Properties, Minimal Multiplicities, and Depth of Fiber Cones}

\author[Jonathan Monta\~{n}o]{Jonathan Monta\~{n}o}
\address{Jonathan Monta\~{n}o \\ Department of Mathematics \\ Purdue University \\150 North University Street \\West Lafayette, IN. 47907}
\email{jmontano@purdue.edu}

\newtheorem{thm}{\bf Theorem}[section]

\newtheorem{prop}[thm]{\bf Proposition}
\newtheorem{lemma}[thm]{\bf Lemma}
\newtheorem{cor}[thm]{\bf Corollary}
\theoremstyle{definition}

\theoremstyle{remark}
\newtheorem{remark}[thm]{\bf Remark}
\newtheorem{example}[thm]{\bf Example}
\numberwithin{equation}{section}

\newcommand{\HH}[3]{\operatorname{H}^{#1}_{#2}(#3)}
\newcommand\Wtilde[1]{\widetilde{#1}}

\def \height{{\operatorname{ht\,}}}
\def \depth{{\operatorname{depth\, }}}
\def \grade{{\operatorname{grade\, }}}

\def\Ass{\operatorname{Ass}}

\def \Im{{\operatorname{Im\,}}}
\def \Ker{{\operatorname{Ker\,}}}
\def \ann{{\operatorname{ann \,}}}
\def \core{{\operatorname{core \,}}}
\def \soc{{\operatorname{soc \,}}}
\def \projdim{{\operatorname{projdim \,}}}

\def\ls{\leqslant}
\def\gs{\geqslant}

\def\ll{\lambda}

\def\fm{\mathfrak{m}}
\def\fM{\mathfrak{M}}
\def\fp{\mathfrak{p}}

\def \AA{\mathbb A}

\def \C{\mathcal C}

\def \G{\mathcal G}
\def \R{\mathcal R}

\def \N{\mathcal N}

\def \K{\mathcal K}
\def \Z{\mathcal Z}

\def \F{\mathcal F}

\begin{document}

\begin{abstract}
Minimal values of multiplicities of ideals have a strong relation with the depth of blowup algebras. In this paper, we introduce the notion of Goto-minimal $j$-multiplicity for ideals of maximal analytic spread. In a Cohen-Macaulay ring, inspired by the work of S.~Goto, A.~Jayanthan, T.~Puthenpurakal, and J.~Verma, we study the interplay among this new notion, the notion of minimal $j$-multiplicity introduced by C.~Polini and Y.~Xie, and the Cohen-Macaulayness of the fiber cone of ideals satisfying certain residual assumptions. We are also able to provide a bound on the reduction number of ideals of Goto-minimal $j$-multiplicity having either Cohen-Macaulay associated graded algebra, or linear decay in the depth of their powers.
\end{abstract}

\maketitle

\section{Introduction}

The fiber cone of an $R$-ideal $I$, defined as the graded algebra $$\F(I)=\bigoplus_{n=0}^{\infty}I^n/I^n\fm,$$ contains asymptotic information about $I$ and its powers. For instance, the Hilbert function of $I$ gives the minimal number of generators of the powers of $I$. An ideal $J\subseteq I$ is a minimal reduction of $I$ if it is minimal with respect to inclusion in the set of all ideals satisfying $I^{n+1}=JI^n$ for all $n\gg 0$. The least $n$ for which this equality occurs for at least one minimal reduction is called the reduction number of $I$. In the case of an infinite residue field $k$, the analytic spread of $I$, which is the dimension of $\F(I)$, is equal to the minimal number of generators of every minimal reduction of $I$. The Cohen-Macaulayness of $\F(I)$ has been studied by several authors using different types of approaches (see for example \cite{CGPU},\cite{CZ},\cite{JV2},\cite{Sh}, \cite{Vi}). Its importance lies in the fact that when $\F(I)$ is Cohen-Macaulay, the images in $I/I\fm=[\F(I)]_0$ of minimal generating sets of minimal reductions of $I$ are maximal regular sequences of $\F(I)$. This makes it possible to reduce to an Artinian algebra while preserving algebraic and homological properties. The main goal of this paper is to investigate the Cohen-Macaulayness of $\F(I)$ for ideals $I$ with minimal values of the $j$-multiplicity.

The $j$-multiplicity was introduced in \cite{AM} as a generalization of the Hilbert-Samuel multiplicity to arbitrary ideals of a Noetherian local ring $(R,\fm)$. The $j$-multiplicity has been a very active research topic in the last few years as several results for $\fm$-primary ideals have been shown to hold for arbitrary ideals if the Hilbert-Samuel multiplicity is replaced by the $j$-multiplicity. For instance, numerical criteria for integral dependence (Rees' criterion, \cite{FM}), combinatorial interpretation of the multiplicity of monomial ideals (\cite{JM}), and the relation with the Cohen-Macaulayness of blowup algebras (\cite{MX}, \cite{PX1}) have been generalized this way.

In \cite{Ab} S.~ Abhyankar proved the well-known lower bound for the Hilbert-Samuel multiplicity of the maximal ideal of a Cohen-Macaulay ring, $e(\fm)\gs \mu(\fm)-d+1$. This inequality has been extended to $\fm$-primary ideals $I$ in two natural ways. The first one is due to G.~Valla and states $$e(I)\gs \ll(I/I^2)-(d-1)\ll(R/I),$$ where the equality occurs if and only if $I^2=JI$  for one, and then every, minimal reduction $J$ of $I$ (see \cite{V1}). The second one is $$e(I)\gs \mu(I)-d+\ll(R/I)$$ with equality if and only if $ I\fm=J\fm$ (see \cite[2.1]{Go1}). The $\fm$-primary ideals $I$ for which the equality holds in either of the above inequalities have been called ideals of minimal multiplicity. Following Rossi and Valla (see \cite{RV1}), we will distinguish between the two possibilities by referring to ideals satisfying $I^2=JI$ as ideals of minimal multiplicity, and to those with $I\fm =J\fm $ as ideals of Goto-minimal multiplicity. 

There is a strong relation between these two notions and the Cohen-Macaulay\-ness of $\F(I)$ and $\G(I)=\bigoplus_{n=0}^{\infty}I^n/I^{n+1}$. If $I$ is of minimal multiplicity, then $\G(I)$ is a Cohen-Macaulay ring (see \cite{S1}, \cite{V1}), whereas if $I$  is of almost minimal multiplicity, i.e., $e(I)= \ll(I/I^2)-(d-1)\ll(R/I)+1$ then  $\depth \G(I)\gs d-1$ (see \cite{RV2}, \cite{Wa}). If the ideal $I$ is of Goto-minimal multiplicity, it was proved in \cite{Go1} that $\G(I)$ is Cohen-Macaulay if and only if $I^2=JI$; in this case $\F(I)$ is also Cohen-Macaulay. In \cite{JV1} under the assumption that $\depth \G(I)\gs d-2$, it was proved that if $I$ has almost Goto-minimal multiplicity, i.e., $e(I)=\mu(I)-d+\ll(R/I)+1$, then $\depth \F(I)\gs d-1$.

The minimal multiplicity property was successfully extended by C. Polini and Y. Xie in \cite{PX1} to arbitrary ideals of analytic spread $d$. They defined the properties of minimal $j$-multiplicity and almost minimal $j$-multiplicity, and with these new notions they were able to generalize the classical results on the Cohen-Macaulayness of $\G(I)$, under certain residual assumptions for $I$. In this paper, we introduce Goto-minimal $j$-multiplicity and almost Goto-minimal $j$-multiplicity and study the interplay among all these notions and the Cohen-Macaulayness of $\F(I)$, $\G(I)$, and the Rees algebra $\R(I)=\bigoplus_{n=0}^{\infty} I^n$, under the same residual assumptions.

We now outline the content of the paper. In Section 2 we will set up the notation for the rest of the paper and will present some relevant theorems that will be used in the proofs of the main results. In this section we also define Goto-minimal $j$-multiplicity and show that the property is well defined.

The goal of Section 3 is to give the definition and present the main implications of the residual assumptions used in the statements of the results. If $R$ is Cohen-Macaulay and $I$ is an ideal of analytic spread $d$, it turns out that if $I$ satisfies $G_d$ and the Artin-Nagata property $AN_{d-2}^-$, then $I$ being of Goto-minimal $j$-multiplicity is equivalent to $I\fm=J\fm$ for a minimal reduction $J$ of $I$ and this is independent of $J$. These residual conditions, $AN_{d-2}^-$ and $G_d$, are satisfied by wide classes of ideals, such as $\fm$-primary ideals, one-dimensional ideals that are generically a complete intersection, perfect ideals of height two that satisfy $G_d$, and perfect Gorenstein ideals of height three that satisfy  $G_d$.  

In Section 4, we will show two of the main results of this paper. The first one characterizes the Cohen-Macaulayness of $\F(I)$ under residual assumptions and a condition weaker than Goto-minimal $j$-multiplicity. The statement is as follows:

\

\noindent {\bf Theorem \ref{FCM}.} {\it Let $R$ be a Cohen-Macaulay local ring of dimension $d$, with maximal ideal $\fm$, and infinite residue field. Let $I$ be an $R$-ideal with analytic spread $\ell(I)=s$ and reduction number $r$.  Assume $I$ satisfies $G_s$ and $AN^{-}_{s-2}$. Let $J$ be a minimal reduction of $I$ such that $r_J(I)=r$ and assume $J\cap I^{j}\fm=JI^{j-1}\fm$ for every $2\ls j\ls r$, then the following are equivalent:
\begin{enumerate}
\item[i)] $\F(I)$ is Cohen-Macaulay.
\item[ii)] $\depth \G(I)\gs s-1.$
\item[iii)] $\depth \R(I)\gs s.$
\end{enumerate}}

\

The second result is a generalization of \cite[2.7]{Go1}, here $a(\F(I))$ is the $a$-invariant of $\F(I)$:

\

\noindent {\bf Theorem \ref{Theo1}.} {\it 
Let $R$ be a Cohen-Macaulay local ring of dimension $d$, with maximal ideal $\fm$, and infinite residue field. Let $I$ be an $R$-ideal with analytic spread $\ell(I)=s$, grade $g$, and reduction number $r$. Assume $I$ satisfies $G_s$, $AN_{s-2}^-$, and $I\fm=J\fm$ for $J$ a minimal reduction of $I$. Consider the following statements 
\begin{enumerate}
\item[i)]  $\R(I)$ is Cohen-Macaulay (when $g\gs 2$),
\item[ii)] $\G(I)$ is Cohen-Macaulay,
\item[iii)] $\F(I)$ is Cohen-Macaulay and $a(\F(I))\ls -g+1,$
\item[iv)] $r\ls s-g+1.$
\end{enumerate}
Then \textnormal{i) } $\Leftrightarrow$ \textnormal{ ii) } $\Rightarrow$ \textnormal{ iii) } $\Rightarrow$ \textnormal{ iv)}. If in addition $\depth R/I^j\gs d-g-j+1$ for every $1\ls j\ls s-g+1$, then all the statements are equivalent.}

\

This theorem in particular provides a bound for the reduction number of $I$ when $\G(I)$ is Cohen-Macaulay or $\grade I\gs 2$ and $\R(I)$ is Cohen-Macaulay. It is important to remark that all the assumptions of these two theorems are satisfied if $I$ is an ideal of Goto-minimal $j$-multiplicity satisfying $G_d$ and $AN_{d-2}^-$. In Section 5 we generalize \cite[4.4]{JV1}, the main result of this section is the following:

\

\noindent {\bf Theorem \ref{theo3}.} {\it 
Let $R$ be a Cohen-Macaulay local ring of dimension $d$ and with infinite residue field.  Let $I$ be an $R$-ideal of analytic spread $\ell(I)=d$ and grade $g$. Assume $I$ satisfies $G_d$, $AN_{d-2}^{-}$, and $\depth R/I^j \gs \min\{d-g-j+1,\, 1\}$ for $j=1$ and 2. If $I$ is of almost Goto-minimal $j$-multiplicity and $\depth \G(I)\gs d-2$, then $\depth \F(I)\gs d-1$.}

\

Finally, in Section 6 we provide several examples illustrating the main results. In this section we also provide a method to construct ideals of Goto-mi\-ni\-mal $j$-multiplicity and almost Goto-minimal $j$-multiplicity. We apply the method to find a family of ideals of Goto-minimal $j$-multiplicity related to perfect ideals of height two that have a presentation matrix of linear forms.

\section{Preliminaries}

Let $(R,\fm, k)$ be a Cohen-Macaulay local ring of dimension $d$, with maximal ideal $\fm$, and infinite residue field $k$. For an ideal $I$ of $R$, we will denote by $\R(I)=\bigoplus_{n=0}^{\infty} I^n$ the {\it Rees algebra of $I$}, by $\G(I)=\bigoplus_{n=0}^{\infty}I^n/I^{n+1}$ the {\it associated graded algebra of $I$}, and by $\F(I)=\bigoplus_{n=0}^{\infty}I^n/I^n\fm $  the {\it fiber cone of $I$}. The dimension of $\G(I)$ is always $d$ and the one of $\R(I)$ is $d+1$ provided $\grade I>0$. The dimension of $\F(I)$ is called the {\it analytic spread} of $I$ and is denoted by $\ell(I).$

An ideal $J\subseteq I$ is a {\it reduction} of $I$ if $I^{n+1}=JI^n$ for all $n\gg 0$, this condition is equivalent to $I$ being in the integral closure of $J$. The smallest $n$ for which the equality is satisfied is denoted by $r_J(I)$, the {\it reduction number} of $I$ with respect to $J$. A reduction is a {\it minimal reduction} if it is minimal with respect to inclusion. The {\it reduction number} of $I$ denoted by $r(I)$, is defined as $\min\{r_J(I): J \text{ is a minimal reduction of }I\}.$ We will write as $\mu(M)$ the minimal number of generators of the $R$-module $M$. Minimal reductions always exist in our setting of infinite residue field and every minimal reduction $J$ of $I$ is minimally generated by $\ell(I)$ elements, i.e., $\mu(J)=\ell(I)$.

The element $x\in I$ is a {\it superficial element} of $I$ if there exists $c>0$ such that $(I^{n+1}:x)\cap I^c=I^n$ for every $n\gs c$. A sequence $x_1,\,\ldots,\,x_t$ is a {\it superficial sequence} of $I$ if $x_i$ is a superficial element of $I/(x_1,\ldots,x_{i-1})$ for every $1\ls i\ls t$. If $I^n$ with $n\gs 1$ is the highest power of $I$ that contains $x$, then we denote by $x^*\in I^n/I^{n+1}=[\G(I)]_n$, the {\it initial form} of $x$. It is a fact that $x$ being a superficial element of $I$ is equivalent to $x^*$ being in $[\G(I)]_1$ and $\ann(x^*)$ being concentrated in finitely many degrees. Superficial elements always exist in our setting of infinite residue field.

The above definitions can be extended to any {\it $I$-good filtration } of $R$, i.e., filtrations $R=I_0\supseteq I_1\supseteq I_2\supseteq \cdots $ such that $I\cdot I_n\subseteq I_{n+1}$ for every $n$ and $I\cdot I_n=I_{n+1}$ for $n\gg 0$ (see \cite[Chapter~1]{RV1}). Of special interest for us is the filtration $\{I^n\fm\}_{n\gs -1}$ where $\fm I^{-1}:=R$, and its {\it associated graded module} $\Z(I)=\bigoplus_{n=0}^{\infty}I^{n-1}\fm/I^{n}\fm$. A superficial element for this filtration is an element $x\in I$ such that there exists $c>0$ such that $(I^{n+1}\fm:x)\cap I^c\fm=I^n\fm$ for every $n\gs c$.
The graded $\G(I)$-module $\Z(I)$ will be useful to relate the depths of  $\F(I)$ and $\G(I)$ through the following proposition. 

If $M=\bigoplus_{n=0}^{\infty} M_n$ is a graded $\G(I)$-module, we will denote by $M_{\gs i} $ the graded submodule $\bigoplus_{n=i}^{\infty} M_n$. In the following proposition all the depths are computed with respect to any ideal of $\G(I)$. 

\begin{prop}\cite[Proposition 5.1]{RV1}\label{ineqdepths} Let $R$ be a Noetherian local ring. Then for every $i$ we have
\begin{enumerate}
\item $\depth \Z(I)_{\gs i} \gs \min\{\depth \G(I)_{\gs i-1}, \depth \F(I)_{\gs i},  \depth \F(I)_{\gs i-1}+1\},$
\item $\depth \F(I)_{\gs i}\gs \min\{\depth \Z(I)_{\gs i},\depth \G(I)_{\gs i-1}+1,\depth \F(I)_{\gs i-1}+2\},$
\item $\depth \G(I)_{\gs i}\gs \min\{\,\depth \Z(I)_{\gs i+1},\depth \F(I)_{\gs i},\depth \F(I)_{\gs i+1}-1\}.$
\end{enumerate}

In particular if $i<0$, we have
\begin{enumerate}
\item[(4)] $\depth \Z(I) \gs \min\{\depth \G(I), \depth \F(I)\},$
\item[(5)] $\depth \F(I)\gs \min\{\depth \Z(I),\depth \G(I)+1\},$
\item[(6)] $\depth \G(I)\gs \min\{\,\depth \Z(I),\depth \F(I)-1\}.$
\end{enumerate}
\end{prop}
\begin{proof}
Let $\N=\bigoplus_{n=0}^{\infty}\fm I^n/I^{n+1}$. From the following exact sequences of $R$-modules 
\[
\begin{tikzcd}
0 \arrow{r} &\fm I^n/I^{n+1} \arrow{r} &I^n/I^{n+1} \arrow{r} &I^n/\fm I^{n} \arrow{r} &0,
\end{tikzcd}
\]
\[
\begin{tikzcd}
0 \arrow{r} &I^n/\fm I^{n} \arrow{r} &\fm I^{n-1}/\fm I^{n} \arrow{r} &\fm I^{n-1}/ I^{n} \arrow{r} &0,
\end{tikzcd}
\]
we obtain exact sequences of graded $\G(I)$-modules

\[
\begin{tikzcd}
0 \arrow{r} &\N_{\gs i} \arrow{r} &\G(I)_{\gs i} \arrow{r} &\F(I)_{\gs i} \arrow{r} &0,
\end{tikzcd}
\]
\[
\begin{tikzcd}
0 \arrow{r} &\F(I)_{\gs i} \arrow{r} &\Z(I)_{\gs i} \arrow{r} &\N_{\gs i-1} \arrow{r} &0,
\end{tikzcd}
\]
for every $i$. The conclusion follows by the depth formula.
\end{proof}

Let $\{a_1,a_2,\ldots,a_n\}$ be a set of generator of $I$. A property $P$ holds for {\it $t$ general elements} in $I$, or for short $P$ is {\it general}, if there exists a Zarisky dense open subset $U$ of $\AA^{nt}_k$ such that for every $(\lambda_{ij})\in R^{nt}$ whose image $(\overline{\lambda_{ij}})$ in $k^{nt}$ lies in $U$, the sequence $x_i=\sum_{j=1}^{n}\lambda_{ij}a_j$ for $i=1,2,\ldots, t$ satisfies the property $P$. The importance of general elements lies in the various properties that are satisfied by sequences of general elements. When the residue field $k$ is infinite, no $k$-vector space is equal to the union of finitely many subspaces, then it is always possible to find a general element in an ideal avoiding a finite collection of sub-ideals. The following properties are also general: 

\begin{enumerate}
\item[i)] Minimal set of generators of $I$,
\item[ii)] sequences of $\grade I$ elements in $I$ that are regular sequences of $R$,
\item[iii)] sequences of $\ell(I)$ elements generating a minimal reduction $J$ of $I$ such that $r_J(I)=r(I)$,
\item[iv)] superficial elements for any $I$-good filtration of $R$.
\end{enumerate}
Our main references for the definitions and facts presented above are \cite{SH} and \cite{RV1}.

Let $T=\HH{0}{\fm\G(I)}{\G(I)}\cong\bigoplus_{n=0}^{\infty}\HH{0}{\fm}{I^n/I^{n+1}}$. Since $T$ is a finitely generated $\G(I)$-module, we can find $u\gg 0$ such that $\fm^uT=0$ and hence $T$ is a $\G(I)/\fm^u\G(I)$-module. Let $P_T(n)$ be the Hilbert polynomial of $T$, then the degree of $P_T(n)$ is $\dim T-1\ls d-1$. The normalized coefficient in degree $d-1$ of $P_T(n)$ is denoted by $j(I)$, i.e.,

$$j(I)=(d-1)!\lim_{n\rightarrow \infty}\frac{\ll(\HH{0}{\fm}{I^n/I^{n+1}})}{n^{d-1}}.$$
This invariant is called the {\it $j$-multiplicity} of $I$.

From the following inequalities $$\dim T\ls \dim \G(I)/\fm^u\G(I) =\dim \G(I)/\fm\G(I)=\ell(I),$$ we obtain that $j(I)\neq 0$ can only happen when $I$ has analytic spread $\ell(I)=d$, in fact these two conditions are equivalent (see \cite[2.1]{NU}). It is easy to see that if $I$ is $\fm$-primary then $T=\G(I)$ and hence $j(I)=e(I)$ the {\it Hilbert-Samuel multiplicity} of $I$. The following theorem provides a method to compute the $j$-multiplicity by reducing the ideal $I$ to a suitable $\fm$-primary ideal.  

\begin{prop}\cite[2.1]{PX1}\label{jmultdim1} Let $R$ be a Noetherian local ring with infinite residue field. Let $I$ be an $R$-ideal of analytic spread $\ell(I)=d$ and $x_1,\,\ldots,\, x_d$ general elements in $I$. Let $\overline{R}:=R/(x_1,\,\ldots,x_{d-1}):I^{\infty}$, then $\overline{R}$ is a Cohen-Macaulay ring of dimension one and $\overline{I}$ is $\overline{\fm}$-primary. Furthermore,
$$j(I)=e(\overline{I})=\ll\left(\overline{R}/(\overline{x_d})\right).$$
\end{prop}

Let $S=\bigoplus_{n=0}^{\infty} S_n$ be a Noetherian standard graded algebra with $S_0$ a Noetherian local ring. For a graded $S$-module $M=\bigoplus_{n=0}^{\infty} M_n$, we will denote by $\depth M$ the depth of $M$ with respect to the irrelevant maximal ideal $\mathfrak{M}$ of $S$. We recall that if $\dim M=s$ then the {\it $a$-invariant} of $M$ is defined as $a(M)=\max\{n: [\HH{s}{\mathfrak{M}}{M}]_n\neq 0\}<\infty.$ That is, the top non-vanishing degree of the top non-vanishing local cohomology module of $M$ with respect to $\mathfrak{M}.$ The $a$-invariant of $\G(I)$ can be used to relate the Cohen-Macaulayness of $\G(I)$ and $\R(I)$:

\begin{thm}\cite[7.1]{IT}\label{IT}
Let $R$ be a local Cohen-Macaulay local ring and $I$ and $R$-ideal of positive height. Then $\R(I)$ is Cohen-Macaulay if and only if $\G(I)$ is Cohen-Macaulay and $a(\G(I))<0.$
\end{thm}

When $I$ is $\fm$-primary and $\G(I)$ is Cohen-Macaulay, a general minimal reduction of $I$ is generated by a regular sequence of $R$ whose images in $\G(I)$ also form a regular sequence. Therefore, $a(\G(I))=\max\{n: I^n\neq I^{n+1}+JI^{n-1}\}-d=r(I)-d$. A similar formula for ideals of analytic spread $\ell(I)=s$ that satisfy condition $G_s$ (see Section 3) is given in the following theorem. 

\begin{thm}\cite[3.5]{SUV1}\label{SUV1}
Let $R$ be a Noetherian local ring with infinite residue field. Let $I$ be an $R$-ideal of height $g$, analytic spread $\ell(I)=s$, and reduction number $r$. Assume $I$ satisfies $G_s$ and $\G(I)$ is Cohen-Macaulay, then $$a(\G(I))=\max\{-g,\,r-s\}.$$  
\end{thm}

In \cite{PX1}, the authors defined an $R$-ideal $I$ of analytic spread $\ell(I)=d$ to be of {\it minimal $j$-multiplicity} if the ideal $\overline{I}$ in Proposition \ref{jmultdim1} is of minimal multiplicity, i.e., $\overline{I}^2=\overline{x_d I}$. They proved that this property is well-defined, i.e., that for general elements $x_1,x_2,\ldots,x_d$ the length $\ll(\overline{I}^2/\overline{x_d I})$ is the same. We will follow this approach to define the notion of Goto-minimal $j$-multiplicity. We first need the following lemma, the prove is similar to \cite[2.3]{PX1}.

\begin{lemma}\label{welldefgoto}
Let $R$ be a Noetherian local ring of dimension $d$, with maximal ideal $\fm$, and infinite residue field. Let $I$ be an $R$-ideal of analytic spread $\ell(I)=d$ and $x_1,\ldots, x_d$ general elements  of $I$. Let $\overline{R}=R/(x_1,\,\ldots,x_{d-1}):I^{\infty}$, then the length $\ell(\overline{ I\fm}/\overline{x_d \fm})$ does not depend on $x_1,\,\ldots,x_d$.
\end{lemma}
 
\begin{proof}
Let $a_1,a_2,\ldots, a_n$ be a set of generators of the ideal $I$, $\{Y_{ij}\}$ a set of $d\times n$ variables, and $S=R[\{Y_{ij}\}]_{\fm R[\{Y_{ij}\}]}$ a faithfully flat extension of $R$.  Set $X_i=\sum_{j=1}^{n}Y_{ij}a_j$ for $i=1,\ldots, d$ and $\overline{S}=S/(X_1,\ldots,X_{d-1}):(IS)^{\infty}$. By Proposition \ref{jmultdim1} and \cite[2.1]{Go1} we have
$$j(I)=j(IS)=e(\overline{IS})\\
          = \ll(\overline{IS}/\overline{I\fm S})-1+\ll(\overline{S}/\overline{IS})+\ll(\overline{I\fm S}/\overline{X_d \fm S}).$$ 
For $\mathfrak{u}=(u_{ij})\in R^{dn}$, let $\phi_{\mathfrak{u}}:S\rightarrow R$ be the ring homomorphism that sends $Y_{ij}$ to $u_{ij}$. Write $x_{i}=\phi_{\mathfrak{u}}(X_{i})=\sum_{j=1}^{n}u_{ij}a_j$ for $i=1,\ldots, d$, then $\phi_{\mathfrak{u}}((X_1,\ldots,X_{d-1}):(IS)^{\infty})\subseteq (x_1,\ldots,x_{d-1}):I^{\infty}.$ By \cite[2.2]{PX1}, there exists an open dense subset $U_1$ of $k^{dn}$ such that if $\overline{\mathfrak{u}}\in U_1$ then 

\begin{align*}
\ll(\overline{IS}/\overline{ I\fm S}) &= \ll(IS/ I\fm  S+(X_1,\ldots,X_{d-1}):_{IS}(IS)^{\infty})\\ 
 & \gs \ll(\phi_{\mathfrak{u}}(IS)/ \phi_{\mathfrak{u}}( I\fm S+(X_1,\ldots,X_{d-1}):_{IS}(IS)^{\infty}))\\
 &=\ll(I/ I\fm +\phi_{\mathfrak{u}}((X_1,\ldots,X_{d-1}):_{IS}(IS)^{\infty}))\\
 &\gs \ll(I/ I\fm+ (x_1,\ldots,x_{d-1}):_I I^{\infty})\\
 &= \ll(\overline{I}/\overline{ I\fm}),
 \end{align*}
where $\overline{R}=R/ (x_1,\ldots,x_{d-1}):I^{\infty}$. Also,
 $
 \ll(\overline{S}/\overline{IS}) \gs \ll(\overline{R}/\overline{I})$ and $\ll(\overline{ I\fm S}/\overline{X_d\fm S}) \gs \ll(\overline{I\fm }/\overline{x_d\fm}).$ Again by Proposition \ref{jmultdim1} and \cite[2.1]{Go1}, there exists an open dense subset $U_2$ of $k^{dn}$, such that if $\overline{\mathfrak{u}}\in U_2$ then $$j(I)=\ll(\overline{I}/\overline{I\fm })-1+\ll(\overline{R}/\overline{I})+\ll(\overline{I\fm }/\overline{x_d I}).$$
 Let $U=U_1\cap U_2$, then if $\overline{\mathfrak{u}}\in U$ we have $\ll(\overline{ I\fm}/\overline{x_d\fm})=\ll(\overline{ I\fm S}/\overline{X_d \fm S})$ which is independent of $\mathfrak{u}.$
\end{proof}

Let $I$ be an ideal of analytic spread $\ell(I)=d$. We define $I$ to be of {\it Goto-minimal $j$-multiplicity} if $\overline{I\fm}=\overline{x_d\fm}$ for general elements $x_1,\ldots,x_d$ in $I$. By Lemma \ref{welldefgoto} this is well defined.

\section{Artin-Nagata Properties}

In this section we will define and present some properties of the residual conditions that will be needed in the subsequent sections. It turns out that the minimal multiplicity properties can be stated in a simpler form under these conditions.

Let $I$ be an $R$-ideal of grade $g$ and let $s\gs g$ be an integer. We say $I$ satisfies $G_s$ if $\mu(I_{\fp})\ls \height \fp$ for every $\fp\in V(I)$ such that $\height \fp < s$. The proper $R$-ideal $K$ is an {\it $s$-residual intersection of $I$} if there exist $s$ elements $x_1\,\,\ldots, \, x_s$ of $I$ such that $K=(x_1\,\ldots,\,x_s):I$ and $\height K \gs s$. We say $K$ is a {\it geometric $s$-residual intersection of $I$} if additionally we have $\height I+K> s.$ The ideal $I$ is said to satisfy the Artin-Nagata property $AN^{-}_{s}$ if the ring $R/K$ is Cohen-Macaulay for every $0\ls i\ls s$ and for every  geometric $i$-residual intersection $K$ of $I$.

\begin{prop}\label{resprop}
Let $R$ be a local Cohen-Macaulay ring of dimension $d$ and with infinite residue field. Let $J\subseteq I$ be $R$-ideals such that $\mu(J)\ls s\ls \height J:I$ for some integer $s\ls d$. Assume $I$ satisfies $G_s$, then
\begin{enumerate}
\item There exists a generating set $x_1,\,\ldots,x_s$ of $J$ such that for every $0\ls i\ls s-1$ and every subset $\{v_1,\ldots, v_i\}$ of $\{1,\ldots,s\}$, we have $\height (x_{v_1},\ldots, x_{v_i}):I\gs i$ and $\height I + (x_{v_1},\ldots, x_{v_i}):I \gs i+1$
\end{enumerate}
\noindent Furthermore, assume $I$ satisfies $AN_{s-2}^{-}$, write $K_i=(x_1,\ldots, x_i):I$, then
\begin{enumerate}
\item[(2)] $R/K_i$ is unmixed of dimension $d-i$ and $x_{i+1}$ is regular in $R/K_i$ for $0\ls i\ls s-1$. 
\item[(3)] $K_i\cap I=(x_1,\,\ldots,x_i)$ and $K_{i}=(x_1,\,\ldots,x_i):x_{i+1}$ for $0\ls i\ls s-1$.
\item[(4)] $\depth R/(x_1,\,\ldots,x_i)=d-i$ and $\depth R/(x_1,\,\ldots,x_i)I \gs \min\{\depth R/I,1\}$ for $0\ls i\ls s-1$.
\item[(5)] $I$ also satisfies $G_s$ and $AN_{s-2}^{-}$ on $R/K_0$.
\item[(6)] If $d\gs s\gs 2$ and $\depth R/I^n\gs d-s+1$, for some $n\gs 1$, then $(x_1) \cap I^{n+1}=x_1I^n$. 
\end{enumerate}
\end{prop}
\begin{proof}
For (1)-(5) see \cite[2.3, 2.4]{JU} and \cite[3.2 (d)]{PX1}. To show (6) notice that by (3), $x_1$ is regular on $I$, therefore $x_1I^n\cong I^n$. Hence $\depth x_1I^n=\depth I^n\gs d-s+2$ and then $\depth R/x_1I^n\gs d-s+1$. The equality will follow once we prove it locally in $\Ass(R/x_1I^n)$, but $\height \fp\ls s-1$ for every $\fp\in \Ass(R/x_1I^n)$ and then by (1) either $I_{\fp}=R_{\fp}$ or $I_{\fp}=(x_1,\ldots,x_{s-1})_{\fp}.$ In the first case the equality holds trivially, for the second case we can proceed as in the end of proof of \cite[4.9]{PX1} to show the equality.
\end{proof} 

\begin{remark}\cite[1.11]{U1}\label{reducht}
Under the assumptions in Proposition \ref{resprop}, if $I$ is an ideal of analytic spread $\ell(I)=s$ satisfying $G_s$ and $AN_{s-2}^-$, the condition $\mu(J)\ls s\ls \height J:I$ is satisfied for every minimal reduction $J$ of $I$. To verify this, it is enough to show $I_{\fp}$ is its own minimal reduction, i.e., $\mu(I_{\fp})\ls \ell(I_{\fp})$, for every $\fp\in V(I)$ with $\height \fp<s$. We show this by induction on $h=\height \fp$. If $h=\height I$, the claim follows from the $G_s$ condition. Assume $\height I<h<s$ and that the claim is true for $h-1$. Let $\fp$ be of height $h$, we can assume $\ell(I_\fp)\ls h-1$. Let $(x_1,\ldots, x_{h-1})_\fp$ be a minimal reduction of $I_\fp$. Suppose $(x_1,\ldots, x_{h-1})_\fp\neq I_\fp$, by induction hypothesis $(x_1,\ldots, x_{h-1})_\fp:I_\fp =h$. By \cite[1.10]{U1} $I_p$ satisfies $AN_{h-1}^-$, then applying Proposition \ref{resprop}, (2) with $s=h$, we conclude $\height (x_1,\ldots, x_{h-1})_\fp:I_\fp=h-1$. We conclude by contradiction that $(x_1,\ldots, x_{h-1})_\fp=I_\fp$. 
\end{remark}

\begin{prop}\cite[3.1]{PX1}\label{genresprop}
Let $R$ be a local Cohen-Macaulay ring with infinite residue field and let $s\ls d$ be an integer. Let $I$ be an $R$-ideal that satisfy $G_s$. Then for general $x_1,\ldots x_s$ in $I$ we have $\height (x_1,\ldots x_s):I\gs s$ and these generators satisfy property \textnormal{(1)} in Proposition \ref{resprop}. 
\end{prop}

The following proposition shows that under certain assumptions, properties $G_s$ and $AN_{s}^-$ specialize after factoring a general element in $\fm$. This technical fact will be used in the proof of Lemma \ref{boundrn}.

\begin{prop}\label{GsANs1}
Let $s\gs0$ be an integer and $I$ and ideal satisfying $G_s$ and $AN_{s}^-$. Assume $d>0$ and $\depth R/I>0$. Let $x$ be a general element in $\fm$ and write $\overline{R}=R/(x)$. Then $\overline{I}$ satisfies $AN_{s}^-$. If $s<d$, then $\overline{I}$ satisfies $G_s$. 
\end{prop}
\begin{proof}
By the depth assumptions, $x$ is a regular element of $R$ and $R/I$. Then by \cite[1.13]{U1}, $\overline{I}$ satisfies $AN_s^-.$ Notice that $G_s$ condition is equivalent to $\height Fitt_{i-1}(I)\gs i$ for every $i\ls s$, where $Fitt_i(I)$ is the {\it $i$th Fitting ideal} of $I$ (see \cite[Proposition 20.6]{E}). Since general $x$ in $\fm$ avoids all the minimal primes of $Fitt_{i-1}(I)$ of height $i$, for $i\ls s$, we have that $\overline{I}$ also satisfies $G_s$.
\end{proof}

In the results of the following sections, we will often require that the ideal $I$ of analytic spread $\ell(I)=s$ satisfies conditions $G_s$ and $AN_{s-2}^{-}$. These assumptions hold true for several well studied classes of ideals, for instance:

\begin{enumerate}
\item[i)]  {\it Equimultiple ideals}, i.e., $\ell(I)=\height I$ (e.g. $\fm$-primary ideals),
\item[ii)] ideals of {\it analytic deviation one}, i.e.,  $\ell(I)=\height I+1$, that are generically complete intersection,
\item[iii)] ideals satisfying $G_s$ and {\it sliding depth} (see \cite[2.1]{JU}),
\item[iv)] in a Gorenstein ring $R$, ideals $I$ satisfying $G_s$ and $\depth R/I^j\gs d-g-j+1$ for every $1\ls j\ls s-g-1$, where $g=\grade I$ (see \cite[2.2]{JU}).
\end{enumerate}

Let $I=(a_1,\ldots,a_n)$ and $H^i=H^i(a_1,\ldots,a_n)$ the Koszul homology modules of $I$. The sliding depth condition in iii) refers to those ideals satisfying the depth inequalities $\depth H^i\gs d-n+i$ for every $0\ls i\ls n-g$. This property is independent of the generating set of $I$. 

This condition holds if all the modules $H^i$ are Cohen-Macaulay. Ideals with this stronger property are called {\it strongly Cohen-Macaulay}. The depth inequalities in iv) hold in particular if $I$ is strongly Cohen-Macaulay satisfying $G_s$. Examples of strongly Cohen-Macaulay ideals are Cohen-Macaulay ideals generated by $n\ls \height I+2$ elements and {\it licci} ideals \cite{Hun} (e.g. perfect ideals of height two and perfect Gorenstein ideals of height three).    

We recall that the {\it core} of $I$, $\core(I)$ is the intersection of all the minimal reductions $J$ of $I$.

\begin{lemma}\label{onlyone}
Let $R$ be a Noetherian local ring, with maximal ideal $\fm$, and infinite residue field. Let $I$ be an ideal of analytic spread $\ell(I)=s$. Assume $I\fm=J\fm$ for a minimal reduction $J$ of $I$. Then

\begin{enumerate}
\item $I\fm=(x_1,\ldots,x_s)\fm$ for general $x_1,\ldots,x_s$ in $I$.
\item  Assume $R$ is Cohen-Macaulay and $I$ satisfies $G_s$ and $AN_{s-1}^-$, then either $\core(I)=I$, or $\core(I)=I\fm$. In particular, $I\fm=J\fm$ for every minimal reduction $J$ of $I$.
\end{enumerate}
\end{lemma}

\begin{proof}

Let $I=(a_1,\ldots, a_n)$ and $J=(b_1,\ldots, b_s)$, where $b_i=\sum_{j=1}^{n}u_{ij}a_j$ for $1\ls i\ls n$ and $(u_{ij})\in R^{sn}$. Let $\{X_{ij}\}$ be a set of $s\times n$ indeterminates, with $1\ls i\ls s$, $1\ls j\ls n$, and $S=R[\{X_{ij}\}].$ Let $J'\subset IS$ be the $S$-ideal $(b_1',\ldots, b_s')$, where $b_i'=\sum_{j=1}^{n}X_{ij}a_j$ for $1\ls i\ls n$. It follows $I\fm S\subseteq J'\fm +(X_{ij}-u_{ij})$ and then $$I\fm S= J'\fm+(X_{ij}-u_{ij})\cap I\fm S.$$ Localizing this equality at the maximal $S$-ideal $\fm'=\fm S+(X_{ij}-u_{ij})$, we obtain $$I\fm S_{\fm'}= J'\fm_{\fm'}+(X_{ij}-u_{ij})_{\fm'}\cap I\fm S_{\fm'}=J'\fm_{\fm'}+(X_{ij}-u_{ij}) I\fm S_{\fm'},$$ where the last equality holds since $\{X_{ij}-u_{ij}\}_{ij}$ is a regular sequence in $S_{\fm'}/I\fm S_{\fm'}$. We conclude by Nakayama's lemma that $I\fm S_{\fm'}=J'\fm_{\fm'}.$ Localizing further at the prime $S$-ideal $\fm S$ we obtain $I\fm S_{\fm S}=J'\fm_{\fm S}.$  Part (1) now follows by standard specialization techniques (see for example \cite[3.1 (a)]{MX}).   

Now, we assume $I$ satisfies $G_s$ and $AN_{s-1}^-$ to show (2). By (1) we have $I\fm\subseteq (x_1,\ldots,x_s)$ an ideal generated by a sequence of $s$ general elements in $I$. By \cite[4.5]{CPU}, $\core(I)$ is equal to a finite intersection of ideals generated by  sequences of $s$ general elements. Therefore, $I\fm\subseteq \core(I)$. If $\mu(I)=s$, then $\core(I)=I$, otherwise given an element in $I\setminus I\fm$, there is a general minimal reduction that does not include it and hence $\core(I)\subseteq I\fm$. For the last statement, if $J$ is a minimal reduction of $I$, then $I\fm=J\cap I\fm=J\fm$.
\end{proof}
 
\begin{thm}\label{mmsimp}
Let $R$ be a local Cohen-Macaulay ring of dimension $d$, with maximal ideal $\fm$, and infinite residue field. Let $I$ be an $R$-ideal with analytic spread $\ell(I)=d$, grade $g$, and  satisfying $G_d$ and $AN_{d-2}^{-}$.
\begin{enumerate}
\item \textnormal{(}\cite[3.4]{PX1}\textnormal{)} Assume $\depth R/I\gs \min\{d-g, 1\}$. Then $I$ is of minimal $j$-mul\-ti\-pli\-ci\-ty if and only if $r(I)\ls 1$.
\item $I$ is of Goto-minimal $j$-multiplicity if and only if $ I\fm=J\fm$ for one, and then every, minimal reduction $J$ of $I$. 
\end{enumerate}
In particular, if $I$ is an $\fm$-primary ideal, then the notions of minimal $j$-multiplicity and Goto-minimal $j$-multiplicity coincide with the classical ones.
\end{thm}
 
\begin{proof}  
Let $x_1,\ldots, x_d$ be general elements in $I$ and $\overline{R}=R/(x_1,\ldots, x_{d-1}):I^{\infty}$. If $d=g$, $I$ is $\fm$-primary and then $(x_1,\ldots,x_{d-1}):I^{\infty}=(x_1,\ldots,x_{d-1})$. If $\overline{I}=I/(x_1,\ldots,x_{d-1})$ is of minimal multiplicity, then $\G(I/((x_1,\ldots,x_{d-1})))$ is Cohen-Macaulay (see \cite[Theorem 2.9]{RV1}), therefore by Sally's machine \cite[Lemma 1.4]{RV1} $\G(I)$ is Cohen-Macaulay. Hence by Valabrega-Valla criterion \cite[Theorem 1.1]{RV1} $(x_1,\ldots, x_{d-1})\cap I^2=(x_1,\ldots, x_{d-1})I$.

By Proposition \ref{genresprop} it follows that either $I_{\fp}=(x_1,\,\ldots,x_{d-1})_{\fp}$ or $I_{\fp}=R_{\fp}$ for every $\fp$ in the punctured spectrum and we have that $\depth R/(x_1,\,\ldots,x_{d-1})>0$. Hence, $$[(x_1,\,\ldots,x_{d-1}):I^{\infty}\cap I]_{\fp}=(x_1,\,\ldots,x_{d-1})_{\fp}$$ for every $\fp\in \Ass(R/(x_1\,\ldots,x_{d-1}))$. Likewise, by Proposition \ref{genresprop} if we assume $\depth R/I>0$ then $\depth R/(x_1,\,\ldots,x_{d-1})I>0$. Hence, $$[(x_1,\,\ldots,x_{d-1}):I^{\infty}\cap I^2]_{\fp}=(x_1,\,\ldots,x_{d-1})I_{\fp}$$ for every $\fp\in  \Ass(R/(x_1\,\ldots,x_{d-1})I).$ It follows that under the assumptions we have $$(x_1,\,\ldots,x_{d-1}):I^{\infty}\cap I^2=(x_1,\,\ldots,x_{d-1})I\,\,\text{   and   }\,\,(x_1,\,\ldots,x_{d-1}):I^{\infty}\cap  I=(x_1,\,\ldots,x_{d-1}).$$ Moreover, since $x_1,\,\ldots,x_{d-1}$ are part of a minimal set of generators of $I$ we have $$(x_1,\,\ldots,x_{d-1}):I^{\infty}\cap  I\fm=(x_1,\,\ldots,x_{d-1})\cap  I\fm=  (x_1,\,\ldots,x_{d-1})\fm.$$ 
From this we conclude that   $\ll(\overline{I^2}/\overline{x_d I})=0$ if and only if $I^2=(x_1,\,\ldots,x_d)I$, and also  $\ll(\overline{ I\fm}/\overline{x_d \fm})=0$ if and only if $I\fm=(x_1,\,\ldots,x_d)\fm.$ 

Now, (1) follows from the fact that $d$ general elements generate a reduction $J$ such that $r_J(I)=r(I)$ (see \cite[8.6.6]{SH}), and (2) follows from \cite[1.9]{U1} and Lemma \ref{onlyone}. 
\end{proof}

\section{Cohen-Macaulay Fiber Cone}
The following lemma will be used in the proof of the first theorem and in one of the corollaries.

\begin{lemma}\label{filtrdepth}
Let $R$ be a Cohen-Macaulay local ring of dimension $d$, with maximal ideal $\fm$, and infinite residue field. Let $I$ be an $R$-ideal with analytic spread $\ell(I)=s$ and satisfying $G_s$ and $AN^{-}_{s-2}$. Let $J=(x_1,\ldots,x_s)$ be a minimal reduction of $I$ with $\{x_1,\ldots,x_s\}$ as in Proposition \ref{resprop} (cf. Remark \ref{reducht}). Let $\{I_n\}_{n\gs 0}$ be one of the two filtrations $\{I^n\}_{n\gs 0}$ or $\{I^n\fm\}_{n\gs -1}$, and $\K$ its associated graded module, i.e., $\G(I)$ or $\Z(I)$ respectively. If $(x_1,\ldots, x_h)\cap I_n=(x_1,\ldots, x_h)I_{n-1}$ for some $1\ls h\ls s$ and every $n\gs 1$. Then $\depth \K\gs h$. 
\end{lemma}
\begin{proof}
Let $N=1$ if $\{I_n\}_{n\gs 0}=\{I^n\}_{n\gs 0}$ and $N=2$ if $\{I_n\}_{n\gs 0}=\{I^n\fm\}_{n\gs -1}$. Let $g=\grade(I)$ and $t=\min\{g,h\}$, we will prove first that $x_1^*,\ldots,x_t^*\in I/I^2=[\G(I)]_1$ is a regular sequence of $\K$. By Proposition \ref{resprop}, (2), $x_1,\ldots, x_t$ is a regular sequence, then by the Valabrega-Valla criterion, we need to prove $ (x_1,\ldots,x_t)\cap I_n=(x_1,\ldots,x_t)I_{n-1}$ for every $n\gs 1$. We will proceed by descending induction on $i$ and by induction on $n$ to show the equality $(x_1,\ldots,x_i)\cap I_n=(x_1,\ldots,x_i)I_{n-1}$ for every $t\ls i\ls h$ and $n\gs 1$. If $i=h$ it holds for every $n$ by assumption. If $1\ls n\ls N$ this holds for every $i$ trivially for $\{I^n\}_{n\gs 0}$ and for $\{I^n\fm\}_{n\gs -1}$ because $\{x_1,\ldots, x_s\}$ is part of a minimal set of generators of $I$. Assume it holds for every $h\gs i>i_0$ and every $n$; and for $i=i_0$ and every $1\ls n<n_0$, with $n_0>N$. We will prove it holds for $i_0$ and $n_0$:
\vskip 0.1in
\noindent \begin{tabular}{llll}
&&$(x_1,\ldots,x_{i_0})\cap I_{n_0}$ &\\
&=&$(x_1,\ldots,x_{i_0})\cap(x_1,\ldots,x_{i_0+1}) \cap I_{n_0}$&\\
&=& $(x_1,\ldots,x_{i_0})\cap (x_1,\ldots,x_{i_0+1})I_{n_0-1}$&\\
&=& $(x_1,\ldots,x_{i_0})I_{n_0-1}+ (x_1,\ldots,x_{i_0})\cap x_{i_0+1}I_{n_0-1}$&\\
&=&$(x_1,\ldots,x_{i_0})I_{n_0-1}+ x_{i_0+1}\big((x_1,\ldots,x_{i_0}):x_{i_0+1}\cap I_{n_0-1}\big)$&\\
&=&$(x_1,\ldots,x_{i_0})I_{n_0-1}+ x_{i_0+1}\big((x_1,\ldots,x_{i_0})\cap I_{n_0-1}\big),$&\hspace{-0.6cm}{ by Proposition \ref{resprop}, (3), since $I_{n_0-1}\subseteq I$,}\\
&=&$(x_1,\ldots,x_{i_0})I_{n_0-1}+ x_{i_0+1}(x_1,\ldots,x_{i_0})I_{n_0-2}$&\\
&$\subseteq$& $(x_1,\ldots,x_{i_0})I_{n_0-1}$.& 
\end{tabular}
\vskip 0.1in
If $h\ls g$ the above argument concludes the proof, then we will assume $h\gs g$. We  proceed by induction on $\sigma(I)= h-g\gs 0$ to show $\depth \K\gs h$, the case $\sigma(I)=0 $ being already covered. We can assume $\sigma(I)\gs 1$ and factoring out $x_1,\ldots,x_g$ we can also assume $g=0$. Let $K_0=\ann I$, all the assumptions are preserved after factoring out $K_0$, i.e., considering the ideal $\overline{I}$ where $\overline{R}=R/K_0$. Indeed, by Proposition \ref{resprop}, (2), (3), and (5), $\overline{R}$ is Cohen-Macaulay of dimension $d$, $\overline{I}$ satisfies $G_s$ and $AN_{s-2}^-$, and since $K_0\cap I=0$ we have $\F(\overline{I})=\F(I)$ and hence $\ell(\overline{I})=s$. The ideal $\overline{J}$ is then a minimal reduction of $\overline{I}$ and in order to show that the sequence $\{\overline{x_1},\ldots,\overline{x_s}\}$ is also as in Proposition \ref{resprop} it is enough to show that if $K\subseteq I$ is any ideal, then $\height \overline{K}:\overline{I}=\height K:I$ and $\height \overline{I} + \overline {K}:\overline{I}=\height I+K:I$. The latter holds because $K_0\subseteq K:I$ and 
$$\overline{K}:\overline{I}=(K+K_0):I/K_0=(K+K_0\cap I):I/K_0=K:I/K_0.$$
Furthermore, $\sigma(\overline{I})<\sigma(I)$ because by Proposition \ref{resprop}, (2) $\grade \overline{I}>0$. By induction hypothesis $\depth \overline{\K}\gs h$, where $\overline{\K}=\bigoplus_{n=0}^{\infty}\overline{I_n}/\overline{I_{n+1}}$. From the following exact sequence   
\[
\begin{tikzcd}
0 \arrow{r} & (I_{N} +K_0)/I_{N}\,[-N+1]\arrow{r} &\K \arrow{r} &\overline{\K} \arrow{r}&0,
\end{tikzcd}
\]
we conclude $\depth\K\gs\min\{\depth\overline{\K}, \depth  (I_{N} +K_0)/I_{N}\}$. By Proposition \ref{resprop}, (3) we have the isomorphism $( I_{N}+K_0)/ I_{N}\cong K_0/K_0\cap  I_{N}\cong K_0$. Since $\overline{R}$ is Cohen-Macaulay it follows $\depth (I_{N}+K_0)/ I_{N}\gs d$. Therefore, $\depth\K\gs \depth\overline{\K}\gs h$, as desired.
\end{proof}

We are now ready to present the first main theorem of this section. The result relates the Cohen-Macaulayness of $\F(I)$ with the depths of $\G(I)$ and $\R(I)$ under an assumption that is weaker than Goto-minimal $j$-multiplicity.

\begin{thm}\label{FCM}
Let $R$ be a Cohen-Macaulay local ring of dimension $d$, with maximal ideal $\fm$, and infinite residue field. Let $I$ be an $R$-ideal with analytic spread $\ell(I)=s$ and reduction number $r$.  Assume $I$ satisfies $G_s$ and $AN^{-}_{s-2}$. Let $J$ be a minimal reduction of $I$ such that $r_J(I)=r$ and assume $J\cap I^{n}\fm=JI^{n-1}\fm$ for every $2\ls n\ls r$, then the following are equivalent:
\begin{enumerate}
\item[i)] $\F(I)$ is Cohen-Macaulay.
\item[ii)] $\depth \G(I)\gs s-1.$
\item[iii)] $\depth \R(I)\gs s.$
\end{enumerate}
In particular, this holds if $s=d$ and $I$ is of Goto-minimal $j$-multiplicity satisfying $G_d$ and $AN_{d-2}^-$.
\end{thm}

\begin{proof}
i) $\Leftrightarrow$ ii): By Proposition \ref{ineqdepths} it suffices to show $\depth \Z(I)\gs s$. This follows by assumption and Lemma \ref{filtrdepth}.

ii) $\Leftrightarrow$ iii): This follows by \cite[3.6 and 3.10]{HM1}.

The last statement follows by Theorem \ref{mmsimp}.
\end{proof}

We obtain from Theorem \ref{FCM} the following corollary, which gives sufficient conditions for $\F(I)$ to be Cohen-Macaulay. The case $r=1$ recovers and generalizes \cite[1]{Sh} and \cite[4.2]{CZ}.

\begin{cor}\label{Theo2}
Let $R$ be a Cohen-Macaulay local ring of dimension $d$, with maximal ideal $\fm$, and infinite residue field. Let $I$ be an $R$-ideal with analytic spread $\ell(I)=s$ and reduction number $r$.  Assume $I$ satisfies $G_s$ and $AN^{-}_{s-2}$. Let $J$ be a minimal reduction of $I$ such that $r_J(I)=r$ and assume the following conditions hold:
\begin{enumerate}
\item $\depth R/I^j\gs d-s+r-j$ for every $1\ls j\ls r-1$,
\item $J\cap I^{n}\fm=JI^{n-1}\fm$ for every $2\ls n\ls r$.
\end{enumerate}
Then $\F(I)$ is Cohen-Macaulay.
\end{cor}
\begin{proof}
By Theorem \ref{FCM} it is enough to show that $\depth \G(I)\gs s-1$. Let $J=(x_1,\ldots,x_s)$ with $\{x_1,\ldots,x_s\}$ as in Proposition \ref{resprop}. The conclusion will follow from Lemma \ref{filtrdepth} once we prove that $(x_1,\ldots,x_{s-1})\cap I^n=(x_1,\ldots,x_{s-1})I^{n-1}$ for every $ n$, we will prove this by induction on $n\gs 0$. If $r\ls 1$, $\depth G(I)\gs s$ by Lemma \ref{filtrdepth}, then we can assume $r\gs 2$. By \cite[2.5 (c)]{JU} and the depth inequalities we have $(x_1,\ldots,x_{s-1})\cap I^n=(x_1,\ldots,x_{s-1})I^{n-1}$ for every $1\ls n\ls r$. Assume it holds for every $n\ls n_0$ with $n_0\gs r$, then: 
\vskip 0.1in
\noindent \begin{tabular}{llll}
&&$(x_1,\ldots,x_{s-1})\cap I^{n_0+1}$ &\\
&=& $(x_1,\ldots,x_{s-1})\cap (x_1,\ldots,x_{s})I^{n_0}$,&\hspace{-0.6cm}{      because $n_0\gs r$,}\\
&=& $(x_1,\ldots,x_{s-1})I^{n_0}+ (x_1,\ldots,x_{s-1})\cap x_{s}I^{n_0}$&\\
&=&$(x_1,\ldots,x_{s-1})I^{n_0}+ x_{s}\big((x_1,\ldots,x_{s-1}):x_{s}\cap I^{n_0}\big)$&\\
&=&$(x_1,\ldots,x_{s-1})I^{n_0}+ x_{s}\big((x_1,\ldots,x_{s-1})\cap I^{n_0}\big),$&\hspace{-0.6cm}{     by Proposition \ref{resprop}, (3), since $n_0>0$,}\\
&=&$(x_1,\ldots,x_{s-1})I^{n_0}+ x_{s}(x_1,\ldots,x_{s-1})I^{n_0-1}$&\\
&$\subseteq$& $(x_1,\ldots,x_{s-1})I^{n_0}$.& 
\end{tabular}

\end{proof}
The next corollary relates the property of minimal $j$-multiplicity and the Cohen-Macaulay property of $\F(I)$. 

\begin{cor}
Let $R$ be a Cohen-Macaulay local ring of dimension $d$ and with infinite residue field. Let $I$ be an $R$-ideal with analytic spread $\ell(I)=d$ and grade $g$. Assume $I$ satisfies $G_d$, $AN^{-}_{d-2}$, and $\depth (R/I)\gs\min\{d-g,\, 1\}$. If $I$ is of minimal $j$-multiplicity, then $\F(I)$ is Cohen-Macaulay.
\end{cor}
\begin{proof}
By Theorem \ref{mmsimp} we have $r=1$ and then the conclusion follows by Corollary \ref{Theo2}.
\end{proof}

An ideal $I$ of analytic spread $\ell(I)=d$ is said to be of {\it almost minimal $j$-multiplicity} if for $d$ general elements $x_1,\ldots, x_d$ of $I$ the ideal $\overline{I}$ of $\overline{R}=R/(x_1,\ldots,x_{d-1}):I^{\infty}$ is of almost minimal multiplicity, i.e., $\ll(\overline{ I^2}/\overline{x_d I})=1$ (cf. \cite{PX1}).

In the next proposition we generalize \cite[3.4]{JPV} to the class of ideals treated in this paper. It provides a characterization of the ideals of almost minimal $j$-multiplicity with Cohen-Macaulay fiber cone. Recall the Hilbert series of $\F(I)$ is defined by $$HS_{\F(I)}(t)=\sum_{n=0}^{\infty}\mu(I^n)t^n.$$

\begin{prop}\label{Theo4}
Let $R$ be a Cohen-Macaulay local ring of dimension $d>0$ and with infinite residue field. Let $I$ be an $R$-ideal with analytic spread $\ell(I)=d$, grade $g$, and reduction number $r$. Assume $I$ satisfies $G_d$, $AN^{-}_{d-2}$, and $\depth (R/I)\gs\min\{d-g,\, 1\}$. If $I$ is of almost minimal $j$-multiplicity, then the following are equivalent:
\begin{enumerate}
\item[i)] $\F(I)$ is Cohen-Macaulay,
\item[ii)]The Hilbert series of $\F(I)$ is 
$$HS_{\F(I)}(t)=\frac{1+(\mu(I)-d)t+t^2+\cdots+t^r}{(1-t)^d},$$
\item[iii)]  $I^2\fm=JI\fm$, for one minimal reduction $J$ of $I$.
\end{enumerate}
\end{prop}

\begin{proof}
By Proposition \ref{resprop}, (2), we can factor out $K_0=\ann I$ and assume $g>0$, as in the proof of Lemma \ref{filtrdepth}. If $g=d$ then $I$ is an $\fm$-primary ideal and this coincides with \cite[3.4]{JPV}. We can assume $\depth R/I\gs 1$. Let $\mathfrak{N}=\G(I)_{\gs 1}$, by \cite[4.10]{PX1} we have $\depth \G(I)\gs d-1$ and $\depth_{\mathfrak{N}} \G(I)_{\gs 1}>0$.

 i) $\Leftrightarrow$ ii): Since $I$ satisfies $G_d$ and $AN_{d-2}^-$, we have $\ll(I^2/(x_1,\ldots, x_d)I)=1$ for $d$ general elements in $I$ (cf. proof of Theorem \ref{mmsimp}), and then this equivalence follows exactly as in the proof of \cite[Theorem 5.5]{RV1}.

 i) $\Rightarrow$ iii): We will proceed by induction on $d$ to show the equality for $J$, an ideal generated by $d$ {\it sequentially general elements} $x_1,\ldots, x_d$ in $I$ (cf. \cite{NU}). The case $d=1$ was covered by the $\fm$-primary case at the beginning of this proof, then we can assume $d\gs 2$. There is a general regular element $x_1\in I$ such that $x_1^*\in I/I^2=[\G(I)]_1$ is regular in $\G(I)_{\gs 1}$. By the Valabrega-Valla criterion we obtain $x_1I\cap I^{n+2}=x_1I^{n+1}$ for every $n\gs 0$. By Proposition \ref{genresprop}, (6), we have $(x_1)\cap I^2=x_1I$, and then we conclude $(x_1)\cap I^{n+1}=x_1I^{n}$ for every $n\gs 0$. In particular this implies $\depth_{\mathfrak{N}} \G(I)> 0$. Let $\overline{R}=R/(x_1)$, it follows that $\F(\overline{I})=\F(I)/(x_1')$ where $x_1'\in I/I\fm=[\F(I)]_1$, and then $\F(\overline{I})$ is also Cohen-Macaulay. Since the ideal $\overline{I}$ is of almost minimal $j$-multiplicity satisfying $G_{d-1}$, $AN_{d-3}^-$, and $\depth \overline{R}/\overline{I}=\depth R/I>0$, by induction hypothesis we obtain $\overline{I^2\fm}=\overline{(x_2,\ldots, x_d)I\fm}$, then $$I^2\fm=(x_2,\ldots, x_d)I\fm +(x_1)\cap I^2\fm=(x_2,\ldots, x_d)I\fm +x_1\fm \cap I^2\fm.$$
By Proposition \ref{ineqdepths} we have $\depth_{\mathfrak{N}}\Z(I)_{\gs 1}>0$, then we can assume $x_1^*$ is also regular in $\Z(I)_{\gs 1}$ and again by Valabrega-Valla we have $x_1\fm\cap I^2\fm=x_1 I\fm.$ Hence, $I^2\fm=(x_1,\ldots, x_d)I\fm$ as desired.

 iii) $\Rightarrow$ i): We proceed as in the proof of Lemma \ref{onlyone} to prove that $I^2\fm=(x_1,\ldots,x_d)I\fm$ for $d$ general elements in $I$. In particular this equality holds for a minimal reduction $J$ such that $r_J(I)=r$ (see \cite[8.6.6]{SH}). The statement now follows by Theorem \ref{FCM}.
\end{proof}

\begin{cor}\label{corminalm}
 Let $R$ be a Cohen-Macaulay local ring of dimension $d>0$ and with infinite residue field. Let $I$ be an $R$-ideal with analytic spread $\ell(I)=d$ and grade $g$. Assume $I$ satisfies $G_d$, $AN^{-}_{d-2}$, and $\depth (R/I)\gs\min\{d-g,\, 1\}$. If $I$ is of Goto-minimal $j$-multiplicity and of almost minimal $j$-multiplicity, then $\F(I)$ is Cohen-Macaulay.
\end{cor}

The following lemma gives an upper bound for the reduction number of $I$ when $\G(I)$ is Cohen-Macaulay. This result will be useful in the proofs of Theorem \ref{Theo1} and Corollary \ref{ineqbound} which in turn can be seen as generalizations of \cite[2.7]{Go1} and \cite[2.8]{Go1}. 

\begin{lemma}\label{boundrn}
Let $R$ be a Cohen-Macaulay local ring, with maximal ideal $\fm$, and infinite residue field. Let $I$ be an $R$-ideal of analytic spread $\ell(I)=s$, grade $g$, and reduction number $r$. Assume $I$ satisfies $G_s$ and $AN_{s-2}^-$. If $\G(I)$ is Cohen-Macaulay and $I\fm =J\fm $ for $J$ a minimal reduction of $I$, then $$r\ls s-g+1.$$
\end{lemma}
\begin{proof}
By Lemma \ref{onlyone} $I\fm=(x_1,\ldots,x_s)\fm$ for general $x_1,\ldots, x_s$ in $I$.  Write $J=(x_1,\ldots, x_s)$. Let $d=\dim R$, if $d=0$, $I^2\subset J=(0)$, hence $r\ls 1$. We can then assume $d>0$. 
 
Also, we can assume $I\neq J$, otherwise $r=0$ and the result follows. After factoring out $K_0=\ann I$ as in the proof of Lemma \ref{filtrdepth}, we can also assume $g>0$. Indeed, if $g=0$, from Proposition \ref{resprop}, (3), $K_0\cap I=0$ then $r$ and $s$ are preserved, whereas  $g$ increases. 

We will first give the proof in the case $s=d$. From $I\fm=J\fm$ we have $I^{n+1}\subset J^n$ for every $n$ and then the following is an exact sequence of $\R(J)$-modules
\[
\begin{tikzcd}
0 \arrow{r} &\C\,[1] \arrow{r} &\G(J) \arrow{r}{\alpha} &\G(I) \arrow{r}{\beta} &\C \arrow{r} &0,
\end{tikzcd}
\]
where $\C$ is the $R(J)$-graded module $\bigoplus_{n=0}^{\infty}I^n/ J^n$. Let $\N=\bigoplus_{n=0}^{\infty}J^n/I^{n+1}$ be the $\R(J)$-module $\Im \alpha =\Ker \beta$. This exact sequence induces the following two short exact sequences

\[
\begin{tikzcd}
0 \arrow{r} &\C\,[1] \arrow{r} &\G(J) \arrow{r} &\N \arrow{r} &0,
\end{tikzcd}
\]
\[
\begin{tikzcd}
0 \arrow{r}  &\N \arrow{r} &\G(I) \arrow{r} &\C \arrow{r} &0.
\end{tikzcd}
\]
By Remark \ref{reducht} and \cite[1.12]{U1}, $J$ satisfies $G_d$ and $AN_{d-2}^-$. Since $r(J)=0$ the algebra $\G(J)$ is Cohen-Macaulay (see for example Lemma \ref{filtrdepth}). Then if $\mathfrak{M}$ is the irrelevant maximal ideal of $\R(J)$, we have $\HH{i}{\mathfrak{M}}{\G(J)}=\HH{i}{\mathfrak{M}}{\G(I)}=0$ for $0\ls i\ls d-1$.

From the first exact sequence we get $\HH{i}{\fM}{\N}\cong\HH{i+1}{\fM}{\C\,[1]}$, for $0\ls i\ls d-2$ and from the second exact sequence we get $\HH{i}{\fM}{\C}\cong\HH{i+1}{\fM}{\N}$, for $0\ls i\ls d-2$. 

We also get $\HH{0}{\fM}{\C}=\HH{0}{\fM}{\N}=0$, then $$\HH{i}{\fM}{\C}=\HH{i}{\fM}{\N}=0,$$ for every $0\ls i\ls d-1$. Again from the exact sequences above, we obtain

\[
\begin{tikzcd}
0 \arrow{r} &\HH{d}{\fM}{\C\,[1]} \arrow{r} &\HH{d}{\fM}{\G(J)} \arrow{r} &\HH{d}{\fM}{\N}\arrow{r} &0
\end{tikzcd}
\]
\[
\begin{tikzcd}
0  \arrow{r} &\HH{d}{\fM}{\N} \arrow{r} &\HH{d}{\fM }{\G(I)} \arrow{r} &\HH{d}{\fM}{\C}\arrow{r} &0
\end{tikzcd}
\]
are exact. Set $a_d(\C)=\max\{n:[\HH{d}{\fM}{\C}]_n\neq 0\}$ and $a_d(\N)=\max\{n:[\HH{d}{\fM}{\N}]_n\neq 0\}$, from the latter exact sequences we obtain $$a (\G(I))=\max\{a_d (\C), a_d (\N)\}\ls \max\{a_d (\C\,[1]), a_d (\N)\}+1=a (\G(J))+1=-g+1,$$ where the last equality follows from $r(J)=0$ and Theorem \ref{SUV1}. Therefore, the inequality $r\ls d-g+1$ follows again by Theorem \ref{SUV1}. 

For the general case consider $\omega(I)=d-s$, we will prove the result by induction on $\omega(I)$. The case $\omega(I)=0$ was given above. Suppose $\omega(I)=d-s>0$ and that the result holds for any ideal with smaller values of $\omega(I)$. Since $\G(I)$ is Cohen-Macaulay and $\dim \G(I)/\fm\G(I)=\dim \F(I)=s$, we have $\grade\fm \G(I)=d-s$. Since the residue field $k$ is infinite, a general element $x$ in $\fm$ is not in any associated prime of $R$ and is not in any of the ideals that are the degree zero component of the associated primes of $\G(I)$. Therefore, $x$ is regular in $R$ and $x^*\in R/I=[\G(I)]_0$ is regular in $\G(I)$. Let  $\overline{R}=R/(x)$, then $\G(\overline{I})=\G(I)/(x^*)$ is Cohen-Macaulay. By \cite[3.3]{EH}, $\depth R/I\gs d-s>0$, then by Proposition \ref{GsANs1}, $\overline{I}$ satisfies $G_s$ and $AN_{s-2}^-$. The ideal $\overline{I}$ has the same $g$, $s$, and $r$ than $I$ and furthermore $\omega(\overline{I})=(d-1)-s=\omega(I)-1$, then the conclusion follows by induction hypothesis.
\end{proof}

We are now ready to present the second main theorem of this section.

\begin{thm}\label{Theo1}
Let $R$ be a Cohen-Macaulay local ring of dimension $d$, with maximal ideal $\fm$, and infinite residue field. Let $I$ be an $R$-ideal with analytic spread $\ell(I)=s$, grade $g$, and reduction number $r$. Assume $I$ satisfies $G_s$, $AN_{s-2}^-$, and $I\fm=J\fm$ for $J$ a minimal reduction of $I$. Consider the following statements 
\begin{enumerate}
\item[i)]  $\R(I)$ is Cohen-Macaulay (when $g\gs 2$),
\item[ii)] $\G(I)$ is Cohen-Macaulay,
\item[iii)] $\F(I)$ is Cohen-Macaulay and $a(\F(I))\ls -g+1,$
\item[iv)] $r\ls s-g+1.$
\end{enumerate}
Then \textnormal{i) } $\Leftrightarrow$ \textnormal{ ii) } $\Rightarrow$ \textnormal{ iii) } $\Rightarrow$ \textnormal{ iv)}. If in addition $\depth R/I^j\gs d-g-j+1$ for every $1\ls j\ls s-g+1$, then all the statements are equivalent.

In particular, this holds if $s=d$ and $I$ is of Goto-minimal $j$-multiplicity satisfying $G_d$ and $AN^{-}_{d-2}$. 
\end{thm}
\begin{proof}

i) $\Leftrightarrow$ ii): By Theorem \ref{IT} if $\R(I)$ is Cohen-Macaulay, then $\G(I)$ is Cohen-Macaulay. The converse holds if $a(\G(I))<0$. This follows by Theorem \ref{SUV1} and Lemma \ref{boundrn}.

ii) $\Rightarrow$ iii): By Lemma \ref{onlyone} (1) and Theorem \ref{FCM}, $\F(I)$ is Cohen-Macaulay. Now, if $\F(I)$ is Cohen-Macaulay we have $a(\F(I))=r-s$. To verify this, since $J\subseteq I$ is a reduction, the images $x_1',\ldots, x_s'$ in $I/I\fm =[\F(I)]_1$ form a regular sequence. Hence, $a(\F(\overline{I}))=a(\F(I)/(x_1',\ldots,x_s'))-s$. The latter algebra is Artinian and therefore 
\begin{align*}
a(\F(I)/(x_1',\ldots,x_s'))&=\max\{j\mid [\F(I)/(x_1',\ldots,x_s')]_j\neq 0\}\\
&= \max\{j\mid I^j\neq JI^{j-1}+I^j\fm \}\\
&= r,
\end{align*}
where the last equality follows by Nakayama's lemma. The inequality now follows by Lemma \ref{boundrn}.

iii) $\Rightarrow$ iv): This follows immediately from $a(\F(I))=r-s$.

If the depth inequalities hold and $r\ls s-g+1$, then $\G(I)$ is Cohen-Macaulay by \cite[3.13]{JU}. Hence all the statements are equivalent. The last statement follows by Theorem \ref{mmsimp}.
\end{proof}

The bound in Lemma \ref{boundrn} is referred as the {\it expected reduction number} in the literature and it has a strong relation with the Cohen-Macaulayness of $\G(I)$ for ideals satisfying certain assumptions (see \cite{PU} and \cite{U2}). In \cite{U2}, Bernd Ulrich introduced the notion of {\it $s$-balanced ideals} to characterize the ideals having the expected reduction number (see \cite[3.1]{U2}), this property for ideals in a Gorenstein ring satisfying $G_s$ and the set of inequalities $\depth R/I^j\gs d-g-j+1$ for every $1\ls j\ls s-g+1$, is equivalent to $J:I$ being independent of the minimal reduction $J$ of $I$. This is the case for ideals $I$ with analytic spread $\ell(I)=s$ satisfying $G_s$, $AN_{s-1}^-$, and such that $I\fm=J\fm$ for a minimal reduction $J$ of $I$ (cf. Lemma \ref{onlyone}).

\begin{cor}\label{ineqbound}
Let $R$ be a Gorenstein local ring of dimension $d$, with maximal ideal $\fm$, and infinite residue field. Let $I$ be an $R$-ideal of analytic spread $\ell(I)=s$, grade $g$, and reduction number $r$. Assume $I$ satisfies $G_s$ and $\depth R/I^j\gs d-g-j+1$ for every $1\ls j\ls s-g+1.$ If $I\fm=J\fm$ for $J$ a minimal reduction of $I$, then one of the following two conditions holds:

\begin{enumerate}
\item[i)] $I=J$ and $r=0$.

\item[ii)] $I\neq J$ and $r=s-g+1$. 
\end{enumerate}
In either case $\G(I)$ and $\F(I)$ are Cohen-Macaulay. Furthermore, if $g\gs 2$ then $\R(I)$ is also Cohen-Macaulay.
\end{cor}
\begin{proof}
Under these assumptions $I$ satisfies $AN_{s-1}^-$ (\cite[2.9]{U1}).
If $I=J$ then clearly $r=0$. If $I\neq J$, since $I$ is $s$-balanced, we have $r= s-g+1$ by \cite[4.8]{U2}. 

The Cohen-Macaulayness of $\R(I)$, $\G(I)$, and $\F(I)$ follow by Theorem \ref{Theo1}.
\end{proof}

\section{Almost Cohen-Macaulay Fiber Cone}
The goal of this section is to generalize \cite[4.4]{JV1} to ideals of analytic spread $\ell(I)=d$ satisfying $G_d$ and $AN_{d-2}^{-}$.

We define $I$ to be of {\it almost Goto-minimal $j$-multiplicity} if for $d$ general elements $x_1,\ldots, x_d$ of $I$ the ideal $\overline{I}$ of $\overline{R}=R/(x_1,\ldots,x_{d-1}):I^{\infty}$ is of almost Goto-minimal multiplicity i.e. $\ll(\overline{ I\fm}/\overline{x_d\fm})=1$.

We will denote by $\tau_J(I)$ the reduction number of $I$ with respect to $J$ in the filtration $\{I^n\fm\}_{n\gs -1}$, i.e., the least integer $n$ such that $I^{n}\fm=JI^{n-1}\fm$ (see \cite[p. 62]{RV1}). For every $j\gs 0$, we will denote by $\Wtilde{I^j}\fm=\bigcup_{t\gs 1}(I^{j+t}\fm:I^{t})$ the {\it Ratliff-Rush filtration} of $I$ with respect to $\fm$ (see \cite[2.2]{JV1}). We present some properties of this filtration that we will need in the proof of the next theorem. These properties are classical, see for example \cite[Lemma 3.1]{RV1} or \cite[2.3]{JV1}.

\begin{prop}\label{proprr}
Let $R$ be a Noetherian local ring, with maximal ideal $\fm$, and infinite residue field. Let $I$ be an $R$-ideal and assume $\grade I>0$, then
\begin{enumerate}
\item $\Wtilde{I^j}\fm= I^{j}\fm$ for $j\gg 0$.
\item $\Wtilde{I^{j}}\fm:x= \Wtilde{I^{j-1}}\fm$ for every $j\gs 1$, for $x$ a general element of $I$.
\end{enumerate}
\end{prop}
\begin{proof}
The first part follows by \cite[2.3 (2)]{JV1}. For the second part, notice that since $x$ regular and a superficial element for the filtration $\{I^n\fm\}_{n\gs -1}$, we have $I^{j+t}\fm:x=I^{j+t-1}\fm$ for $t\gg 0$. We can also assume $\Wtilde{I^j}\fm=I^{j+t}\fm:I^{t}$ and $\Wtilde{I^{j-1}}\fm=I^{j+t-1}\fm:I^{t}$. Therefore,
$$\Wtilde{I^{j}}\fm:x=(I^{j+t}\fm:I^{t}):x=(I^{j+t}\fm:x):I^{t}=I^{j+t-1}\fm:I^{t}=\Wtilde{I^{j-1}}\fm.$$
\end{proof}

\begin{lemma}\label{fibdim2} Let $R$ be a Cohen-Macaulay  local ring of dimension $d\gs 2$, with maximal ideal $\fm$, and infinite residue field.  Let $I$ be an $R$-ideal of analytic spread $\ell(I)=d$ satisfying $G_d$ and $AN_{d-2}^{-}$. Let $J=(x_1,\ldots, x_d)$ be an ideal generated by $d$ general elements in $I$ and $\mathfrak{N}=\G(I)_{\gs 1}$. If $\ll(I\fm/J\fm)=1$, then 
\begin{enumerate}
\item If $d=2$, $\depth \Z(I)>0$.  
\item $\depth_{\mathfrak{N}} \Z(I)_{\gs n}\gs d-1$, for every $n\gs 2$.
\item If $\grade I\gs d-1$, then $\depth_{\mathfrak{N}} \Z(I)\gs d-1.$ 
\end{enumerate}
\end{lemma}
\begin{proof}

We first assume $d=2$ and $I$ has positive grade, we will prove $\depth_{\mathfrak{N}} \Z(I)>0$ which will show (3) for $d=2$. Since $I\fm$ is generated by elements of the form $ab$ with $a\in I$ and $b\in \fm$ we have, $ I\fm= J \fm+ (ab)$ for some $a\in I$ and $b\in \fm$ such that $ab\not\in  J\fm$ but $ ab\fm\in J\fm$, then the multiplication map 
\[
\begin{tikzcd}
I\fm/J\fm \arrow[two heads]{r}{a^{j-1}} & I^{j}\fm/JI^{j-1}\fm
\end{tikzcd}
\]
is surjective for every $j\gs 1$. Then, $\ll( I^{j}\fm/JI^{j-1}\fm)\ls 1$ for every $j\gs 1$. Let $K_1=(x_1):x_2$ and $\overline{R}=R/K_1$. Let $\overline{I}$ be the image of $I$ in $\overline{R}$. By Proposition \ref{resprop}, (2), $x_2$ is regular in $\overline{R}$, then $\overline{I}$ is an $\overline{\fm}$-primary ideal. Let $s$ be $\tau_{\overline{J}}(\overline{I})$. We will show that $\tau_J(I)=s$. 
By Propositions \ref{genresprop} and \ref{proprr}, (1), the Hilbert functions of the two $\overline{I}$-good filtrations $\{\overline{I^n\fm}\}_{n\gs -1}$ and $\{\overline{\Wtilde{I^n}\fm}\}_{n\gs -1}$ are asymptotically the same and therefore their first Hilbert coefficients are equal. By \cite[Lemma 2.2]{RV1} this implies $\sum_{j\gs 1}\ll\big(\overline{I^j\fm}/\overline{x_2I^{j-1}\fm}\big)=\sum_{j\gs 1}\ll\big(\overline{\Wtilde{I^j}\fm}/\overline{x_2\Wtilde{I^{j-1}}\fm}\big)$ because $\overline{x_2}$ is a superficial element for the two filtrations above. We conclude $\sum_{j\gs 1}\ll\big(\overline{\Wtilde{I^j}\fm}/\overline{x_2\Wtilde{I^{j-1}}\fm}\big)=s-1$, since $\ll\big(\overline{I^j\fm}/\overline{x_2I^{j-1}\fm}\big)=1$ for every $1\ls j<s$. Hence,

\vskip 0.1in

\noindent \begin{tabular}{llll}
\hspace{1.0cm}$s-1$&=&$\sum_{j\gs 1}\ll\big(\Wtilde{I^j}\fm/x_2\Wtilde{I^{j-1}}\fm
+K_1\cap \Wtilde{I^j}\fm \big)$&\\
   &=& $\sum_{j\gs 1}\ll\big(\Wtilde{I^j}\fm/x_2\Wtilde{I^{j-1}}\fm
+(x_1)\cap \Wtilde{I^j}\fm \big),$&{ by Proposition \ref{resprop}, (3),}\\
&=& $\sum_{j\gs 1}\ll\big(\Wtilde{I^j}\fm/x_2\Wtilde{I^{j-1}}\fm
+(x_1)(\Wtilde{I^j}\fm:x_1) \big)$&\\
&=& $\sum_{j\gs 1}\ll\big(\Wtilde{I^j}\fm/x_2\Wtilde{I^{j-1}}\fm
+x_1\Wtilde{I^{j-1}}\fm\big),$&{ by Proposition \ref{proprr}, (2),}\\
&=& $\sum_{j\gs 1}\ll\big(\Wtilde{I^j}\fm/J\Wtilde{I^{j-1}}\fm\big)$.&
\end{tabular}
\vskip 0.1in
 This shows in particular that all these lengths are finite. By \cite[Corollary 4.2]{RV1} we have $\tau_J(I)\ls k+\mu_R(N)$ where $k=\min\{j\mid I^{j}\fm\subseteq J\Wtilde{I^{j-1}}\fm\}$ and $N=\bigoplus_{j\gs 1}\Wtilde{I^{j}}\fm/J\Wtilde{I^{j-1}}\fm+I^{j}\fm.$ These numbers are finite by Proposition \ref{proprr}, (1). We remark that in \cite{RV1} this bound is stated for $\fm$-primary ideals, however the proof is not restricted to this case.  Now, 
\vskip 0.1in

\noindent \begin{tabular}{llll}
\hspace{0.7cm}$\mu_R(N)$&=&$\sum_{j\gs 1}\mu_R\big(\Wtilde{I^{j}}\fm/J\Wtilde{I^{j-1}}\fm+I^{j}\fm\big)$&\\
&$\ls$&$\sum_{j\gs 1}\ll\big(\Wtilde{I^{j}}\fm/J\Wtilde{I^{j-1}}\fm+I^{j}\fm\big)$&\\
&=&$\sum_{j\gs 1}\Big(\ll\big(\Wtilde{I^{j}}\fm/J\Wtilde{I^{j-1}}\fm\big)-\ll\big(J\Wtilde{I^{j-1}}\fm+I^{j}\fm/J\Wtilde{I^{j-1}\fm}\big)\Big)$&\\
&=& $\sum_{j\gs 1}\ll\big(\Wtilde{I^{j}}\fm/J\Wtilde{I^{j-1}}\fm\big)-(k-1)$&\\
&=&$s-k$.&
\end{tabular}
\vskip 0.1in
 Therefore, $\tau_J(I)\ls k+(s-k)=s$, as desired. From the equality $\tau_J(I)=\tau_{\overline{J}}(\overline{I})$ together with the fact $\ll( I^{j}\fm/JI^{j-1}\fm)\ls 1$,  we obtain that for every $j\gs 1$:
\vskip 0.1in

\noindent \begin{tabular}{llll}
\hspace{1.5cm}0&=&$\ll\big(I^{j}\fm/JI^{j-1}\fm\big)-\ll\big(\overline{I^j\fm}/\overline{x_2I^{j-1}\fm}\big)$&\\
 &=&$\ll\big(I^{j}\fm/JI^{j-1}\fm\big)-\ll\big(\overline{I^j\fm}/\overline{JI^{j-1}\fm}\big)$&\\
 &=&$\ll\big(I^{j}\fm/JI^{j-1}\fm\big)-\ll\big(I^{j}\fm/JI^{j-1}\fm+K_1\cap I^{j}\fm\big)$&\\
 &=& $\ll\big(JI^{j-1}\fm+K_1\cap I^{j}\fm/JI^{j-1}\fm\big)$&\\
 &=&$\ll\big(JI^{j-1}\fm+(x_1)\cap I^{j}\fm/JI^{j-1}\fm\big),$&\hspace{-1.05cm}{ by Proposition \ref{resprop}, (3),}
\end{tabular}
\vskip 0.1in
 Therefore,$(x_1)\cap I^{j}\fm\subseteq JI^{j-1}\fm,$ for every $j\gs 1$. We will show now that  $x_1^*\in I/I^2=[\G(I)]_1$  is regular in $\Z(I)$, which will prove our claim. Since $x_1$ is regular in $R$, by the Valabrega-Valla criterion we only need to show $(x_1)\cap I^{j}\fm =x_1I^{j-1}\fm$ for $j\gs 1$. We prove this by induction on $j$. It is clear for $j=1$ since $x_1$ is part of a minimal set of generators of $I$, then we can assume $j\gs 2$. From the containment $(x_1)\cap I^{j}\fm\subseteq JI^{j-1}\fm$ we obtain:

\vskip 0.1in

\noindent \begin{tabular}{llll}
$(x_1)\cap I^{j}\fm$ &=& $(x_1)\cap JI^{j-1}\fm$&\\
&=& $x_1 I^{j-1}\fm+(x_1)\cap x_2 I^{j-1}\fm$&\\
&=& $x_1 I^{j-1}\fm+x_2((x_1):x_2\cap I^{j-1}\fm)$&\\
\end{tabular}

\noindent \begin{tabular}{llll}
\hspace{1.55cm}\,&=&$x_1 I^{j-1}\fm+x_2((x_1)\cap I^{j-1}\fm),$&\hspace{1.88cm}{ by Proposition \ref{resprop}, (3),}\\
&=& $x_1 I^{j-1}\fm+x_2x_1 I^{j-2}\fm,$&\hspace{1.88cm}{ by induction hypothesis,}\\
 &=& $x_1I^{j-1}\fm,$&
\end{tabular}
\vskip 0.1in
\noindent as desired. 

Now, let $I$ be of arbitrary grade. Let $K_0=\ann I$ and $\overline{R}=R/K_0$, then from the fact $K_0\cap I=0$ we get as in the proof of Lemma \ref{filtrdepth},  $\depth \Z(I)\gs \depth \Z(\overline{I})$. Notice $\ll(\overline{I\fm}/\overline{J\fm})=1$ and since $\F(I)=\F(\overline{I})$, $\ell(\overline{I})=2$ . Furthermore, by Proposition \ref{resprop}, (2), (5), the ideal $\overline{I}$ satisfies $G_2$ and $AN_0^-$ and $\grade \overline{I}>0$. Therefore $\depth \Z(\overline{I})>0$ and part (1) follows. Since $\Z(I)_{\gs n}=\Z(\overline{I})_{\gs n}$ for every $n\gs 2$, we also have part (2) for $d=2$.

Let $d\gs 3$, $K_{d-2}=(x_1,\ldots, x_{d-2}):I$, and $\overline{R}=R/K_{d-2}$. By Proposition \ref{resprop}, (2), (5), $\dim \overline{R}=2$ and $\overline{I}$ satisfies all the assumptions. Hence, from Proposition \ref{resprop}, (3), $\depth_{\mathfrak{N}} \Z(I/(x_1,\ldots, x_{d-2}))_{\gs n}=\depth_{\mathfrak{N}} \Z(\overline{I})_{\gs n}>0$ for every $n\gs 2$, and by Sally's machine \cite[Lemma 1.4]{RV1}, $\depth_{\mathfrak{N}} \Z(I)_{\gs n}\gs d-1$.

If $\grade I\gs d-1$, then $K_{d-2}=(x_1,\ldots, x_{d-2})$ and $\depth_{\mathfrak{N}}\Z(I/(x_1,\ldots, x_{d-2}))>0$. Again by Sally's machine we conclude $\depth_{\mathfrak{N}} \Z(I)\gs d-1$.
\end{proof}

The following is the main theorem of this section.

\begin{thm}\label{theo3}
 Let $R$ be a Cohen-Macaulay local ring of dimension $d$ and with infinite residue field.  Let $I$ be an $R$-ideal of analytic spread $\ell(I)=d$ and grade $g$. Assume $I$ satisfies $G_d$, $AN_{d-2}^{-}$, and $\depth R/I^j \gs \min\{d-g-j+1,\, 1\}$ for $j=1$ and 2. If $I$ is of almost Goto-minimal $j$-multiplicity and $\depth \G(I)\gs d-2$, then $\depth \F(I)\gs d-1$.
\end{thm}
\begin{proof}
Let $\mathfrak{N}=\G(I)_{\gs 1}$. If $d=g$, the result follows by \cite[Theorem 5.2]{RV1}, and if $d=g+1$ then it follows by Proposition \ref{ineqdepths}, (5) and Lemma \ref{fibdim2}, (3). Therefore we can assume $\min\{\depth R/I,\, \depth R/I^2\}>0$. As in the proof of Theorem \ref{mmsimp}, we can show that under these assumptions almost Goto-minimal $j$-multiplicity is equivalent to $\ll(I\fm/J\fm)=1$, for $J=(x_1,\ldots, x_d)$ an ideal generated by $d$ general elements in $I$. We will proceed by induction on $d$. The case $d=2$ follows by Proposition \ref{ineqdepths}, (5), and Lemma \ref{fibdim2}, (1). Assume that $d\gs 3$ and that the result holds in dimension $d-1$. As in the proof of Lemma \ref{fibdim2} we can reduce to the case $\grade I>0$ by factoring out $K_0=\ann I$, indeed if $\overline{R}=R/K_0$, then $\F(I)=\F(\overline{I})$ and from the fact that $K_0$ is Cohen-Macaulay it follows $\depth\G(\overline{I})\gs d-2$ and $\min\{\depth R/I+K_0,\depth R/I^2+K_0\}>0$. 

By Proposition \ref{ineqdepths} and Lemma \ref{fibdim2}, (2), we have $\depth \F(I)_{\gs n}\gs 2$, for every $n\gs 2$. But $\fm\G(I)\subseteq \ann \F(I)_{\gs n}$, therefore $\depth_{\mathfrak{N}}\F(I)_{\gs n}=\depth \F(I)_{\gs n}$, and again by Proposition \ref{ineqdepths} we obtain $\depth_{\mathfrak{N}} \G(I)_{\gs 2}>0$. Since $\G(I)_{\gs 2}$ is the associated graded module of the module $I^2$ with respect to the $I$-adic filtration, i.e., $\bigoplus_{n=0}^{\infty} I^n I^2 /I^{n+1} I^2$, we obtain by the Valabrega-Valla criterion that for a general $x$ in $I$ we have $x I^2\cap I^{n+3}=xI^{n+2}$ for every $n\gs 0$. Now, by Proposition \ref{resprop} (6), $(x)\cap I^2=xI$ and $(x)\cap I^3=xI^2$ and hence by the equalities above, $$(x)\cap I^{n+1}=xI^{n}$$ for every $n\gs 0$. Since $x$ is regular in $R$, we conclude $x^*\in I/I^2=[\G(I)]_1$ is regular in $\G(I)$.

The element $x$ is a superficial element for the two filtrations $\{I^n\}_{n\gs 0}$ and $\{I^n\fm\}_{n\gs -1}$ (cf. \cite[8.5.7]{SH}). Hence, $x$ is a superficial element for $\G(I)$ and $\F(I)$ in the sense of \cite{JV2}. Let $\overline{R}=R/(x)$. The ideal $\overline{I}$ satisfies $G_{d-1}$, $AN_{d-3}^-$, and $\ell(\overline{I})=d-1$. Since $(x)\cap  I\fm=x\fm$, $\overline{I}$ is also of almost Goto-minimal $j$-multiplicity. From $\G(\overline{I})\cong\G(I)/(x^*)$ we conclude $\depth \G(\overline{I})\gs d-3$ and then all the assumptions are satisfied by $\overline{I}.$ By induction hypothesis $\depth \F(\overline{I})\gs d-2>0$ and hence by Sally machine for fiber cones (see proof of \cite[2.7]{JV2}), $x'\in I/I\fm= [\F(I)]_1$ is regular. Again from $(x)\cap I^{n+1}=xI^{n}$ for every $n\gs 0$, we obtain $\F(\overline{I})=\F(I)/(x')$ and then $\depth \F(I)=1+\depth \F(\overline{I})\gs d-1$, as desired.
\end{proof}

\section{Examples}
The goal of this section is to provide examples of ideals of Goto-minimal $j$-multiplicity satisfying the Artin-Nagata properties assumed in the results of the previous sections.

\begin{example}
The monomial ideal $$I=(x_1^2,x_1x_2,\ldots, x_1x_d, x_2^2,x_2x_3,\ldots,x_2x_n)$$ in the ring $k[[x_1,\ldots,x_d]]$ with $k$ an infinite field is a {\it strongly stable} ideal of height 2. By \cite[3.1 and 3.3]{S} $I$ has analytic spread $\ell(I)=d$ and  satisfies $G_d$ and $AN_{d-2}^-$. By \cite[5.1]{S}, $\core(I)=I\fm$ and then by Theorem \ref{mmsimp}, $I$ is of Goto-minimal $j$-multiplicity. By \cite[4.9]{S}, $r(I)\ls 1$, then the algebra $\G(I)$ is Cohen-Macaulay. By Theorem \ref{Theo1},  $\R(I)$ and $\F(I)$ are Cohen-Macaulay.
\end{example}

\begin{example} Let $R=k[[x,y,z,w]]$ with $k$ and infinite field and $\fm$ the irrelevant maximal ideal of $R$. Let $M$ be the matrix 

\[
M=\begin{pmatrix}
x&y&z&w\\
w&x&y&z
\end{pmatrix}.
\]
Let $I=I_2(M)$ be the ideal generated by the 2-minors of $M$. This ideal was studied in \cite[4.14]{PX1}, $R/I$ is a one dimensional Cohen-Macaulay ring and $I$ has analytic spread $\ell(I)=4$. Also, $I$ is generically a complete intersection, then it satisfies $G_4$ and $AN_{2}^-$.  We can use Macaulay2 \cite{GS} to see that $I\neq J$ and $I\fm=J\fm$ for one (and then every) minimal reduction $J$ of $I$. Then $I$ is of Goto-minimal $j$-multiplicity. We conclude by Corollary \ref{ineqbound} that $r(I)=2$, and that the algebras $\R(I)$, $\G(I)$, and $\F(I)$ are Cohen-Macaulay. The ideal $I$ is of almost minimal $j$-multiplicity (see \cite[4.14]{PX1}), then the Cohen-Macaulayness of $\F(I)$ also follows by Corollary \ref{corminalm}.
\end{example}

\begin{example} The assumption $\depth \G(I)\gs d-2$ is necessary for the conclusion $\depth \F(I)\gs d-1$ in Theorem \ref{theo3}. This can be observed from \cite[4.6]{JV1}. Let $R=k[[x,y,z]]$ with $k$ an infinite field, $I=(-x^2+y^2,-y^2+z^2,xy,yz,zx)$, and $J=(-x^2+y^2,-y^2+z^2,xy)$. The ideal $J$ is a reduction of the $\fm$-primary ideal $I$ and $\ll(I\fm/J\fm)=1$, then $I$ is of almost Goto-minimal $j$-multiplicity. It can be shown that $\depth \G(I)=0$ and $\depth \F(I)=1$.    
\end{example}

\begin{example}\label{ex4}
Let $R=k[[x_1,\ldots,x_d]]$ with $k$ an infinite field, $d\gs 1$, and $\fm$ be the maximal ideal of $R$. Let $I$ be an $R$-ideal with analytic spread $\ell(I)=d$, grade $g$, and reduction number $r$. Assume $I$ satisfies $G_d$ and the depth inequalities $\depth R/I^j\gs d-g-j+1$ for $1\ls j\ls d-g$. If $I$ has a presentation matrix consisting of linear forms and $\mu(I)=d+1$, then by \cite[3.9]{CPU2} $\core(I)=I\fm$. Hence $I$ is of Goto-minimal $j$-multiplicity and by Corollary \ref{ineqbound} we have $r=d-g+1$. Also, $\G(I)$ and $\F(I)$ are Cohen-Macaulay. If $g\gs 2$, $\R(I)$ is Cohen-Macaulay as well.
\end{example}

 The following proposition provides a way of constructing ideals of Goto-mi\-ni\-mal $j$-multiplicity and almost Goto-minimal $j$-multiplicity starting from an ideal that is known to satisfy $G_d$ and $AN_{d-2}^-$ (e.g. strongly Cohen-Macaulay ideals satisfying $G_d$).

\begin{prop}\label{exprop}
Let $R$ be a Cohen-Macaulay ring of dimension $d$, with maximal ideal $\fm$, and infinite residue field. Let $I$ be an $R$-ideal with analytic spread $\ell(I)=d$ and satisfying $G_d$ and $AN_{d-2}^-$. Let $J$ be a minimal reduction of $I$. 
\begin{enumerate}
\item Let $K$ be an $R$-ideal such that $J\subseteq K\subseteq (J:\fm)\cap I $. Then $K$ is of Goto-minimal $j$-multiplicity satisfying $G_d$ and $AN_{d-2}^-$.
\item Assume $R$ contains a field. Let $a_1,\ldots, a_n$ be a set of minimal generators of $\fm$. Let $K$ be as in (1) and $a\in I\cap (J:\fm^2)\cap (J:a_1,\ldots,a_{n-1})\setminus (J:a_n)$. Let $H=K+(a)$, then $H$ is of almost Goto-minimal $j$-multiplicity satisfying $G_d$ and $AN_{d-2}^-$. 
 \end{enumerate}
\end{prop}
\begin{proof}
Let $L$ be any ideal with $J\subseteq L\subseteq I$. From $I\subseteq \overline{J}$, the integral closure of $J$, we conclude $J$ is a minimal reduction of $L$, therefore $\ell(L)=d$. By Remark \ref{reducht} $J_{\fp}=L_{\fp}=I_{\fp}$ in the punctured spectrum, then $L$ satisfies $G_d$. By \cite[1.12]{U1}, $L$ satisfies $AN_{d-2}^-$. This shows that $J$ is a minimal reduction of $K$ and $H$, and that both satisfy $G_d$ and $AN_{d-2}^-$.

From $J\cap I\fm=J\fm$, we obtain $K\fm=J\fm$ and $\ll(H\fm/J\fm)=1$. By Theorem \ref{mmsimp}, (2) we conclude $K$ is of Goto-minimal $j$-multiplicity. If $R$ contains a field, by \cite[3.1 (b)]{MX} and Lemma \ref{onlyone}, (2) we have $\ll(H\fm/(x_1,\ldots x_d)\fm)= 1$ for general $x_1,\ldots,x_d$ in $K$, hence $H$ is of almost Goto-minimal $j$-multiplicity.  
\end{proof}

\begin{prop}\label{superex}
Let $R=k[[x_1,\ldots,x_d]]$ with $k$ an infinite field and $d\gs 3$. Let $I$ be a perfect $R$-ideal of height two satisfying $G_d$, with analytic spread $\ell(I)=d$, and $\mu(I)=n$. Assume $I$ has a presentation matrix of linear forms. After performing elementary row operations we can assume the $R$-ideal $J$ generated by the first $d$ generators is a reduction of $I$. Let $K$ be an $R$-ideal such that $J\subseteq K\subseteq I$, then $K$ is of Goto-minimal $j$-multiplicity if and only if $$K\subseteq I\fm^{n-d-1}+J.$$ 
For any of these ideals $K$, we have $\R(K)$, $\G(K)$, and $\F(K)$ are Cohen-Macaulay.
\end{prop}
\begin{proof}
 Since $R/I$ is Cohen-Macaulay we have $J:\fm\subseteq I$ and then by Proposition \ref{exprop}, (1) and Theorem \ref{mmsimp}, (2) it would be enough to show $J:\fm=I\fm^{n-d-1}+J$.

We can assume $\mu(I)\gs d+1$. Let $A$ be an $n\times (n-1)$ presentation matrix of $I$ consisting of linear forms. Let $B$ be the $(n-d)\times (n-1)$ matrix consisting of the last $n-d$ rows of $A$, which is a presentation matrix of the module $I/J$. By Remark \ref{reducht}, $\ann I/J$ is $\fm$-primary, therefore $Fitt_0(I/J)=I_{n-d}(B)$ is an $\fm$-primary ideal (see \cite[Proposition 20.7]{E}). Since $B$ has linear entries, we have $I_{n-d}(B)=\fm^{n-d}$ and hence $I\fm^{n-d}\subseteq J$. This shows $I\fm^{n-d-1}+J\subseteq J:\fm$ and since $I$ is generated in degree $n-1$, the equality would follow once we prove $(J:\fm)/J=\soc(R/J)$ is generated by forms of degree $2n-d-2$.

For this notice that $\soc(R/J)=\soc(I/J)$ and the latter module is presented by $B$. It follows from $\grade I_{n-d}(B)=d=(n-1)-(n-d)+1$, that the Buchsbaum-Rim complex $\C^1$ of $B$ is exact (cf. \cite[Section A2.6.1]{E}). The module $I/J$ is generated in degree $n-1$ and every entry in $A$ is linear, hence the last shift of $\C^1$ is $-2(n-1)$. Since $\projdim I/J=d$, it is well known that the latter implies $\soc(I/J)$ is generated in degree $2(n-1)-d$ which proves the claim.

If $n=d+1$ the Cohen-Macaulayness of $\R(K)$, $\G(K)$, and $\F(K)$ follow by Example \ref{ex4}, so we can assume $n\neq d+1$. As $K$ satisfies $G_d$ and $AN_{d-2}^-$, by Theorem \ref{Theo1} and \cite[3.8]{PX1} it suffices to show $r(K)\ls 1$. We can assume $n\gs d+2$ and then $2n-2d-2\gs n-d$. It follows that 
$$(I\fm^{n-d-1})^2=I^2\fm^{2n-2d-2}\subseteq IJ$$ and hence $$I^2\fm^{2n-2d-2}\subseteq IJ\cap\fm^{4n-2d-4}=IJ\fm^{2n-2d-2}\subseteq J^2.$$
Let $K=K'+J$ where $K'\subseteq I\fm^{n-d-1}$. Then $$K^2=K'^2+K'J+J^2=K'J+J^2=JK$$which finishes the proof.
\end{proof}

\section*{Acknowledgements}

The author thanks Yu Xie for conversations that led him to initiate this project. He also would like to thank Mark Johnson for helpful conversations. The author thanks his PhD advisor, Bernd Ulrich, for his constant support and enlightening discussions. He is also grateful to the referee for her or his helpful corrections.

\end{document}